\documentclass[12pt]{amsart}
\usepackage{hyperref}
\usepackage{amsmath,amssymb,latexsym,amsthm}
\usepackage{enumerate}
\usepackage{url}
\usepackage{graphicx,color}
\usepackage{psfrag}

\def\mltext{} 

\newtheorem{theorem}{Theorem}
\newtheorem{lemma}[theorem]{Lemma}
\newtheorem{proposition}[theorem]{Proposition}

\newtheorem{corollary}[theorem]{Corollary}

\def\R{\mathbb R}

\def\Z{\mathbb Z}

\def\p{\partial}

\def\hat{\widehat}
\def\tilde{\widetilde}

\def\det{\text{det}}

\def \beq {\begin {eqnarray}}
\def \eeq {\end {eqnarray}}
\def \ba {\begin {eqnarray*}}
\def \ea {\end  {eqnarray*}}

\newcommand{\Ein}{\hbox{Ein}}
\newcommand{\Ric}{\hbox{Ric}}

\title[Determination  of the spacetime]{Determination  of the spacetime from local time measurements
}
\author{Matti Lassas}
\address{Matti Lassas, University of Helsinki, P.O. Box 68 FI-00014}
\email{Matti.Lassas@helsinki.fi}

\author{Lauri Oksanen}
\address{Department of Mathematics, University College London, Gower Street, London UK, WC1E 6BT.}
\email{l.oksanen@ucl.ac.uk}

\author{Yang Yang}
\address{Department of Mathematics, Purdue University, West Lafayette, IN 47907, USA}
\email{yang926@purdue.edu}
\date{\today}

\begin{document}

\begin{abstract}
We consider an inverse problem for a Lorentzian spacetime $(M,g)$,
and show that time measurements, that is, the knowledge of 
the Lorentzian time separation function on a submanifold $\Sigma\subset M$ determine
the $C^\infty$-jet of the metric in the Fermi coordinates associated to $\Sigma$.
We use this result to study the global determination of the spacetime $(M,g)$
when it has a real-analytic structure
or is stationary and satisfies the Einstein-scalar field equations. In addition to this,
we require that $(M,g)$ is geodesically complete modulo scalar curvature singularities.
The results are Lorentzian counterparts of extensively studied inverse problems in Riemannian
geometry - the determination of the jet of the metric and the boundary rigidity problem. 
We give also counterexamples
in cases when the assumptions are not valid, and discuss inverse problems in general relativity. 
\end{abstract}

\maketitle

\section{introduction}

{Inverse problems for hyperbolic equations have been studied extensively using a 
geometric point of view, see e.g.\ \cite{AKKLT,BeK,Eskin, E2011,KKL,KrKL,LO}. This is due to the fact that 
for a hyperbolic equation with time-independent coefficients, the travel time of the waves
between two points defines a natural Riemannian distance between these
points. The corresponding Riemannian metric is called the travel time metric.
A classical inverse problem is to determine the wave speed
inside the object given the travel times between
the boundary points, or equivalently, the distances between the boundary points.														
In this paper we study geometric inverse problems for Lorentzian manifolds, that are  related to 
hyperbolic equations with time-depending coefficients and to general relativity.

Before formulating the geometric inverse problems for Lorenzian manifolds
that we will study, let us recall earier results
for Riemannian manifolds.} A paradigm problem is the boundary rigidity problem: 
does the restriction $\hat d|_{\partial \hat M\times\partial \hat M}$ of the Riemannian distance function $\hat d : \hat M\times \hat M\rightarrow\mathbb{R}$
 determine uniquely a Riemannian manifold with boundary $(\hat M,\hat  g)$.
If this is possible, then $(\hat M,\hat g)$ is said to be boundary rigid.
Since the boundary distance function takes into account only the shortest paths, it is easy to construct counterexamples where $\hat d|_{\partial \hat M\times\partial \hat M}$ does not carry information on an open subset of $M$. Thus some a-priori conditions on $(\hat M,\hat g)$ are necessary for boundary rigidity. 

Michel has conjectured that simple manifolds are boundary rigid \cite{M}.
We recall that a compact Riemannian manifold with boundary $(\hat M,\hat g)$ is simple, if $\partial \hat M$ is strictly convex 
and if for any $x\in\hat  M$ the exponential map 
$\exp_x$
is a diffeomorphism. 
Pestov and Uhlmann proved the conjecture in the dimension two \cite{PU}  but it is open in higher dimensions.

A related problem to determine the $C^{\infty}$-jet of the metric tensor $\hat g$ on the boundary from
the Riemannian distance function $\hat d|_{\partial \hat M\times\partial \hat M}$
was solved for simple Riemannian manifolds in \cite{LSU}. 
Here we extend this result for Lorentzian manifolds. Our main motivation comes from the theory of relativity, whence we consider  
{a Lorentzian manifold  without boundary, and replace 
the restriction of the Riemannian distance function} 
with the time separations between points on a timelike hypersurface.

Let us suppose that $(M,g)$ is a Lorentzian manifold without boundary. The two main theorems of the paper concern determination of $(M,g)$ given time separations between points on a timelike hypersurface $\Sigma$. Our first result is of local nature: we show that the time separations determine the $C^\infty$-jet of the metric tensor $g$ at a point $x_0\in\Sigma$ assuming that there are many timelike geodesics starting near $x_0$ and intersecting $\Sigma$ again later. The result is obtained by adapting the method developed in \cite{SU} and \cite{LSU} to the Lorentzian context. The method has been previously used only in the Riemannian setting. Our global result is that $C^\infty$-jet of the metric tensor at a point determines the universal Loretzian covering space of $(M,g)$ assuming that $(M,g)$ is real-analytic and geodesically complete modulo scalar curvature singularities, see the definition before Theorem 2.


\subsection{Previous literature}

The boundary distance rigidity question in the Lorentzian context has been studied by Anderson, 
  Dahl,  and Howard  \cite{And-D-H} 
  who have studied slab-like manifolds and show, in particular, that flat  two dimensional product manifolds
  are boundary rigid, that is, the Lorentzian distances of boundary points determine the manifold uniquely
 under natural assumptions. 
In the Riemannian case, in addition to the above mentioned paper by Pestov and Uhlmann \cite{PU}, boundary rigidity has been proved 
for subdomains of Euclidean space \cite{Gr}, for subspaces of an open hemisphere in two dimension \cite{M}, for subspaces of symmetric spaces of constant negative curvature \cite{BCG}, for two dimensional spaces of negative curvature \cite{C1,O}. It was shown in \cite{SU3} that metrics a priori close to a metric in a generic set, which includes real-analytic metrics, are boundary rigid, and in \cite{LSU} it was shown that two metrics with identical boundary distance functions differ by an isometry which fixes the boundary if one of the metrics is close to the Euclidean metric. For other results see \cite{BI,CDS,PSU,SU1}.

{In  \cite{PU}, the Riemannian boundary rigidity problem is reduced to an inverse problem for the Laplace-Beltrami equation
on a two-dimensional manifold. The solution of this problem  is heavily based on
the use of the underlying real-analytic conformal structure that the Riemannian surfaces
have \cite{LU,LTU,LeU}. In the present paper we will use similar kind of   underlying real-analytic structure 
to study inverse problems for the 
stationary spacetimes.}

In addition to \cite{LSU}, the deteremination of the $C^\infty$-jet of the Riemannian metric tensor has been studied 
in \cite{SU}, where the problem of this type was considered for a class of non-simple manifolds, and the authors showed that knowledge of the lens data in a neighborhood of a boundary point determines $C^{\infty}$-jet of the metric at this point. The boundary distances $\hat d|_{\p \hat M \times \p \hat M}$
determine the lens data in the case of a simple Riemannian manifold.

\section{Statement of the results}

Let $(M,g)$ be a $(1+n)$-dimensional smooth manifold $M$ with a Lorentzian metric $g$ of signature $(-,+,\dots,+)$. We recall that a simply convex neighborhood is an open subset $\mathcal U$ in $M$ which is a normal neighborhood for every point inside. It is well-known that in a Lorentzian manifold, each point has a simply convex neighborhood, see e.g. \cite{SW}. 

We will begin by formulating a local result on a simply convex neighborhood $\mathcal U \subset M$.
We regard $(\mathcal{U},g)$ as Lorentzian manifold and define the causality relation as usual, that is, for $x$ and $y$ in $\mathcal{U}$, we write $x\ll y$ if there is a future-pointing timelike curve $\mu([0,l])$, $l>0$ in $\mathcal{U}$ from $x$ to $y$.
{We emhasize   that  a simply convex neighborhood $(\mathcal{U},g)$ is time-oriented and hence $x \ll y$ does not hold when $x=y$.}
Let us define two open subsets of $\mathcal{U}$
$$I^{+}_{\mathcal{U}}(x):=\{z\in \mathcal{U}:x\ll z\}, \quad\quad\quad I^{-}_{\mathcal{U}}(y):=\{z\in \mathcal{U}:z\ll y\}.$$
We call $I^{+}_{\mathcal{U}}(x)$ the chronological future of $x$ in $\mathcal{U}$ and $I^{-}_{\mathcal{U}}(y)$ the chronological past of $y$ in $\mathcal{U}$.

On the Lorentzian manifold $(\mathcal{U},g)$, we can define the Lorentzian distance function, also called the time separation function, 
$$d:\mathcal{U}\times\mathcal{U}\rightarrow\mathbb{R}$$ as follows. For any $x,y \in \mathcal{U}$, by simply convexity of $\mathcal{U}$, there exists a unique geodesic $\gamma_{x,y}:[0,1]\rightarrow \mathcal{U}$ connecting them, that is, 
$\gamma_{x,y}(0)=x$ and $\gamma_{x,y}(1)=y$.
We define
$$L(\gamma_{x,y}):=\displaystyle\int^{1}_{0}|\dot{\gamma}_{x,y}(t)|_{g}\,dt,$$
where $|v|_{g}=|(v,v)_{g}|^{\frac{1}{2}}$ and $(v,w)_{g}=g_{jk}(x)v^jw^k$ is the scalar product 
of vectors $v,w\in T_xM$
with respect to the metric tensor $g$. The Lorentzian distance function is defined as
\begin{align}
\label{def_timesep}
d(x,y)=\left\{
\begin{array}{cl}
L(\gamma_{x,y}), & \textrm{ $\gamma_{x,y}$ is timelike and future-pointing, } \\
0, & \textrm{ otherwise.}\\
\end{array}\right.
\end{align}
Note that this function encodes the causality information and thus is not symmetric.

\begin{figure}[b]
\begin{center}
\includegraphics[height=0.3\textheight]{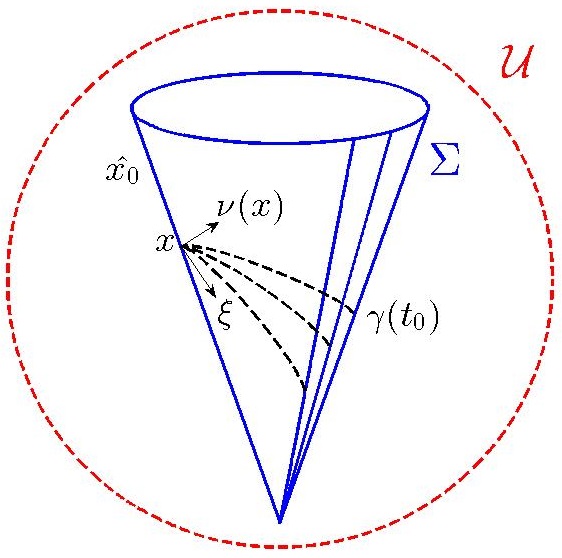}
\end{center}
\caption{An example of $\Sigma$ in $\mathbb{R}^{2+1}$. Here $\Sigma$ is a cone minus the tip.} 
\end{figure}

We will assume that $d$ is known on an oriented smooth timelike open submanifold $\Sigma \subset \mathcal U$
of codimension $1$.
Moreover, we assume that the topological closure of $\Sigma$ is compact and satisfies $\overline{\Sigma}\subset\mathcal{U}$.
Suppose $\hat x_0\in\Sigma$ and $T_{\hat x_0}\Sigma$ is a timelike subspace of $T_{\hat x_0} M$, let $\nu$ be a unit normal vector field of $\Sigma$ near $\hat x_0$. We say that $\Sigma$ is {\em timelike convex} near the point $\hat x_0 \in \Sigma$ and a timelike vector
$\hat\xi_0 \in T_{\hat x_0} \Sigma$ in the direction $\nu(\hat x_0)$, if the following hypothesis \textbf{H} holds.
%

\begin{description}
\item[\textbf{H}]{\em
There is an open neighborhood $U$ of $(\hat x_0, \hat \xi_0)$ in $T\Sigma$ satisfying the following:
for any $(x,\xi)\in U$, there is $\epsilon>0$ such that if $r\in (0,\epsilon)$, then the geodesic $\gamma(t)$ with
      $$\gamma(0)=x, \quad \dot{\gamma}(0)=\xi+r\nu(x),$$
satisfies $\gamma(t_0)\in\overline{\Sigma}$ for some $t_0>0$, and $\gamma(t) \in\mathcal{U}$ for $t\in (0,t_{0})$.
}
\end{description}

A physically motivated example of $\Sigma$ can be found in Section 3.5, 
which consists of {union of the world lines} of freely falling material particles issued from a fixed point with identical Newtonian speeds. Our local result asserts that the knowledge of the Lorentzian distance function $d$ on a timelike hypersurface $\Sigma$ uniquely determines the $C^{\infty}$-jet of the Lorentzian metric $g$ at a point $\hat x_0 \in \Sigma$ assuming that $\Sigma$ is timelike convex near $(\hat x_0, \hat \xi_0)$ for some timelike vector $\hat\xi_0\in T_{\hat x_0}\Sigma$.


\begin{theorem}
\label{thm_local}
Let $(M, g)$ and $(\widetilde{M},\widetilde{g})$ be two smooth Lorentzian manifolds, let $\Sigma \subset M$ and $\widetilde \Sigma \subset \widetilde{M}$
be smooth timelike submanifolds of codimension $1$ such that their closures are compact in simply convex neighborhoods $\mathcal U$ and $\widetilde{\mathcal U}$ respectively.
Let $\hat x_0 \in \Sigma$ and let $\hat \xi_0 \in T_{\hat x_0} \Sigma$ be timelike.
Suppose that there is a diffeomorphism $\Phi : \Sigma \to \widetilde \Sigma$,
such that $\Phi_{\ast\hat x_0}(\hat \xi_0)$ is timelike, and such that 
$\Sigma$ and $\widetilde \Sigma$ are timelike convex near $(\hat x_0, \hat \xi_0)$
in the direction of $\nu(\hat{x_0})$ and near $(\Phi(\hat x_0), \Phi_{\ast\hat x_0}(\hat \xi_0))$ in the direction of $\tilde{\nu}(\Phi(\hat{x_0}))$ respectively.
Suppose, furthermore, that the corresponding Lorentzian distance functions satisfy
$$d(x,y)=\tilde{d}(\Phi(x),\Phi(y)), \quad\quad\quad \hbox{for all }x,y\in\Sigma.$$
Then the $C^{\infty}$-jet of $g$ at $\hat x_0$ coincides with the $C^{\infty}$-jet of $\tilde{g}$ at $\Phi(\hat x_0)$.
\end{theorem}

Here $\Phi_{\ast\hat x_0}(\hat\xi_0)\in T_{\Phi(\hat x_0)}\widetilde{\Sigma}$ is the image of $\hat\xi_0$ under the push-forward $\Phi_{\ast}$ at $\hat{x_0}$. 
A way to formulate the equality of the $C^\infty$-jets is to say that all the derivatives of the metric tensors are equal in suitable coordinates.
In the proof we use semigeodesic coordinates (also called Fermi coordinates) associated to $\Sigma$ and $\widetilde \Sigma$
when they are identified by using the diffeomorphism $\Phi$, see the paragraphs before the proof of Theorem \ref{thm_local} and \eqref{normal} for the details.

We say that $(M,g)$ is real-analytic if the manifold $M$ has an real-analytic structure with respect to which the metric tensor $g$ is real-analytic. 
%
{We note that when
 $(M,g)$ is real-analytic, the determination of the $C^\infty$-jet of the metric tensor
 does not imply global uniqueness results without additional assumptions even if we assumed a priori that M is simply connected.
The reason for this is that there can be multiple incompatible real-analytic extensions of a real-analytic manifold, as is seen in Example \ref{ex_incompatible} below.
However, we will show that for a  real-analytic manifold $(M,g)$
the  determination of the $C^\infty$-jet
 implies a global uniqueness result via analytic continuation under a topological completeness assumption that we will describe next.}

We recall that a function $\iota : M \to \R$ is a scalar curvature invariant if it is of the form
$$
\iota(x) = I(g(x), R(x), \nabla R(x), \dots, \nabla^k R(x)), \quad k=1,2,\dots,
$$
where $I$ is a smooth function, $g$ is the metric tensor, $R$ is the corresponding curvature tensor, and $\nabla$ stands for covariant differentiation.
For example, the Kretschmann scalar, written in local coordinates as $R_{abcd} R^{abcd}$, is a scalar valued curvature invariant of the form
$\iota(x) = I(g(x), R(x))$.
The Kretschmann scalar for the Schwarzschild black hole is a constant times $r^{-6}$ where $r$ is the radial coordinate.

We say that $(M,g)$ is geodesically complete modulo scalar curvature singularities if every maximal geodesic $\gamma : (\ell_-,\ell_+) \to M$ satisfies
$\ell_\pm = \pm\infty$ or there is a scalar curvature invariant $\iota$ such that $\iota(\gamma(t))$ is unbounded as $t \to \ell_\pm$. We will show the following global result.

\begin{theorem}
\label{thm_global}
Let $(M, g)$ and $(\widetilde{M},\widetilde{g})$ be two smooth Lorentzian manifolds satisfying the assumptions of Theorem \ref{thm_local}.
Suppose, furthermore, that $(M,g)$ and $(\tilde M,\tilde g)$ are connected, geodesically complete modulo scalar curvature singularities
and real-analytic.
{Then the universal Lorentzian covering spaces of 
$(M,g)$ and $(\tilde M,\tilde g)$ are isometric.}
\end{theorem}

Let us emphasize that the result is sharp in the sense that 
only the universal covering space, and not the manifold itself, can be determined.
For instance, the Minkowski space $\R^{1+3}$ and the flat torus ${\mathbb T}^{1+3}:=\R^{1+3}/\Z^4$ with the Minkowski metric
contain isometric subsets $(0,1)^4$,
and the time measurements on a small submanifold $\Sigma \subset (0,1)^4$
are identical in both cases. 

\section{Examples}

\subsection{Riemannian manifolds}

Let us illustrate the relation between Theorem \ref{thm_global} and the earlier results for Riemannian manifolds discussed in the introduction.

Let  $(\hat M,\hat g)$ be a real-analytic complete Riemannian manifold of dimension $n$, and let $\hat d$ be the
 Riemannian distance function on $\hat M$.
The
completeness could be replaced by an assumption similar to
 geodesically completeness modulo scalar curvature singularities.
Let us consider the 
product manifold
$M=\R \times \hat M$ with the Lorentzian metric 
\begin{align}
\label{prduct_metric}
g(t,x)=-dt^2 + \hat  g(x), \quad (t,x) \in \R \times \hat M,\quad \hat  g(x)=
\hat g_{jk}(x)dx^jdx^k.
\end{align}
Let $\hat {\mathcal U}\subset
 \hat M$ be a simply convex open set and let $\hat \Sigma\subset \hat {\mathcal U}$ be an $n-1$ dimensional submanifold. 
 Let $\hat \nu (y)$ be a Riemannian normal vector of $\hat \Sigma$ at $y\in \hat \Sigma$. For some fixed $(\hat{y}_0,\hat{\eta}_0)\in T\hat{\Sigma}$, we assume the following is valid: 
\medskip

${\bf H2:}$
 There is an open neighborhood $\hat{U}\subset T\hat\Sigma$ of $(\hat{y_0},\hat{\eta_0})$ in $T\hat\Sigma$ satisfying the following:
for any $(y,\eta)\in \hat{U}$, there is $\epsilon>0$ such that if $r\in (0,\epsilon)$, then the Riemannian 
geodesic $\hat \gamma(s)$ with
      $$\hat \gamma(0)=y, \quad \p_s{\hat \gamma}(0)=\eta+r\hat \nu (y),$$
satisfies $\hat \gamma(s_0)\in\hat{\Sigma}$ for some $s_0>0$, and $\hat{\gamma}(s) \in\hat {\mathcal{U}}$ for $s\in (0,s_{0})$.
\medskip

\begin{figure}[h]
\begin{center}
\includegraphics[height=0.3\textheight]{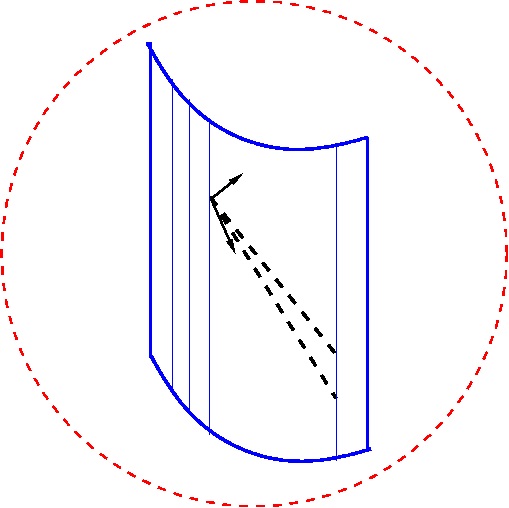}
\end{center}
\caption{A product manifold $M=\R \times \hat M$ with Lorentzian metric $g(t,x)=-dt^2 + \hat  g(x)$. Here $(\hat M, \hat g)$ is a real-analytic complete Riemannian manifold.} 
\end{figure}

In fact, the hypothesis ${\bf (H2)}$ on the Riemannian manifold $\hat{M}$ implies (\textbf{H}) on the Lorentzian manifold $M$. To see this, consider the point $(0,\hat{y_0})\in M$. We can choose $c_0\in\mathbb{R}$ sufficiently large so that $c_0\partial_{t}|_0+\hat{\eta_0} \in T_{(0,\hat{y_0})}M $ is timelike. Let $\pi:M\rightarrow\hat{M}$ be the canonical projection $\pi(t,y)=y$, then $\pi$ maps a geodesic $\gamma$ in $M$ to the Riemannian geodesic $\pi\circ\gamma$ in $\hat{M}$. In particular, $\pi$ projects the geodesic $\gamma$ in $M$ with the initial data
$$\gamma(0)=(t_0,y), \quad \dot{\gamma}(0)=c\partial_t|_{t_0}+\eta+r\pi^\ast\hat{\nu}(t_0,y)$$
to the geodesic $\pi\circ\gamma$ in $\hat{M}$ with the initial data
$$\pi\circ\gamma(0)=y,\quad \dot{(\pi\circ\gamma)}(0)=\eta+r\hat{\nu}(y).$$
Here $\pi^\ast\hat{\nu}$ is the lift to $\mathbb{R}\times\hat{\Sigma}$ of the vector field $\nu$. Thus if $\epsilon>0$ is chosen as in ${\bf (H2)}$, we have from ${\bf (H2)}$ that $\pi\circ\gamma(s_0)\in\hat{\Sigma}$ for some $s_0>0$ and $\pi\circ\gamma(s)\in\hat{\mathcal{U}}$ for $s\in (0,s_0)$. Consequently for large enough $\ell>0$, we conclude $\gamma(s_0)\in (-\ell,\ell)\times\hat{\Sigma}$ and $\gamma(s)\in (-\ell,\ell)\times\hat{\mathcal{U}}$ for $s\in (0,s_0)$. Finally, we choose $c$ to be sufficiently close to $c_0$ so that $\gamma$ is still timelike, and choose $t_0$ to be sufficiently close to $0$. Putting these together, we have shown that $(-\ell,\ell)\times\hat{\Sigma}$ is timelike convex near $((0,\hat{y}_0), c_0\partial_t |_{0}+\hat{\eta}_0)$. (here we use $(-\ell,\ell)\times\hat{\Sigma}$ instead of $\mathbb{R}\times\hat{\Sigma}$ since the former has compact closure, see the condition before (\textbf{H})). The simply convex neighborhood in (\textbf{H}) can be taken to be $\mathcal{U}:=(-2\ell,2\ell)\times\hat{\mathcal{U}}$.

 
In particular, if $\hat{N}\subset \hat{M}$ is an open set that has a strictly convex smooth boundary, then any
$\hat{y}_0\in \p \hat N$ has a simply convex neighborhood $\hat {\mathcal U}\subset
 \hat M$ and  $\hat{\Sigma}:= \hat {\mathcal U} \cap \partial\hat N$ satisfies the assumption ${\bf (H2)}$. That $\hat{\Sigma}$ satisfies ${\bf (H2)}$ simply follows from the strict convexity of $\partial\hat{N}$. As a result of the analysis in the previous paragraph, the submanifold $(-\ell,\ell)\times\hat{\Sigma}$ is timelike convex near $((0,\hat{y}_0),c_0\partial_t|_{0}+\hat{\eta}_0)$, that is, (\textbf{H}) is satisfied.

The restriction of the Riemannian distance function
 $\hat d|_{\hat \Sigma\times\hat \Sigma}$ determines the restriction of the Lorentzian distance
 function 
 $ d|_{\Sigma\times \Sigma}$ by
\begin{align}
  \label{def_timesep Riemann}
d((t_1,x_1)\, ,\, (t_2,x_2))=\left\{
\begin{array}{cl}
\sqrt{(t_2-t_1)^2-\hat d(x_1,x_2)^2 }, & \hbox{if }t_2-t_1>\hat d(x_1,x_2), \\
0, & \textrm{ otherwise.}\\
\end{array}\right.
\end{align}
  Thus, if we are given $\hat \Sigma$ and 
 $ \hat d|_{ \hat\Sigma\times  \hat\Sigma}$, we can determine
by Theorems \ref{thm_local} and \ref{thm_global},   the universal covering space of the Lorentzian manifold $(M,g)$.

The manifold $(M,g)$ is stationary (in fact, static) spacetime with the Killing field $Z=\frac \p{\p t}$
that corresponds to the ``direction of time''. Observe that there may be several Killing
fields, as can be seen considering the standard Minkowski space $\R^{1+3}$. All elements
in the Lorentz group $O(1,3)$ define an isometry of $\R^{1+3}$ that may change the time axis to
any timelike line.
In Section \ref{sec_stationary_analytic} we consider also determination of the Killing field in a stationary spacetime.

\subsection{Stationary spacetimes satisfying Einstein-scalar field equations}
\label{sec_stationary_analytic}

We will apply Theorem  \ref{thm_global} to the Einstein-scalar field model. For related inverse problems for the same model,
see \cite{KLS-Einstein1,KLS-Einstein2}. 

Let $M$ be a $(1+3)$ dimensional manifold.
Let us recall the Einstein field equation $\Ein(g) = T$, where the Einstein tensor $\Ein(g)$ is defined by
\begin{align*}
\Ein_{jk}(g) = \Ric_{jk}(g)-\frac 12 S(g)\,g_{jk}, \quad j,k=0,1,\dots,3,
\end{align*}
$\Ric$ is the Ricci curvature tensor, $S$ is the scalar curvature and $T$ is a stress-energy tensor.
If $T=0$ and $(M,g)$ is a solution to the Einstein field equation, then $(M,g)$ is called a vacuum spacetime.
We recall that a Lorentzian manifold $(M,g)$ is stationary if there is a timelike vector field $Z$ satisfying
\begin{align}
\label{def_Killing}
\mathcal L_Zg=0,
\end{align}
where $\mathcal L_Z$ is the Lie derivative with respect to the vector field $Z$.
The vector fields $Z$ satisfying (\ref{def_Killing}) are called Killing fields {and when
(\ref{def_Killing}) is valid, we say that $g$ is stationary with respect to the Killing field $Z$}.

Below, we consider the Einstein field equations with scalar fields $\phi=(\phi_\ell)_{\ell=1}^L$,
\begin{align} \label{eq: Ein1}
& \Ein_{jk}(g)+\Lambda g_{jk}=T_{jk}(g,\phi),\\
\label{eq: Ein2}
& T_{jk}(g,\phi)=\bigg(\sum_{\ell=1}^L\p_j\phi_\ell \,\p_k\phi_\ell
-\frac 12 g_{jk}g^{pq}\p_p\phi_\ell\,\p_q\phi_\ell\bigg)-{\mathcal V}( \phi)
g_{jk},
\\
\label{eq: Ein3}
& \square_g\phi_\ell+\mathcal V^\prime_\ell(\phi)=0.
\end{align}
Here $\Lambda \in \R$ is the cosmological constant and $\mathcal V:\R^L\to \R$ is a real-analytic function that physically corresponds
to the potential energy of the scalar fields,
e.g., $\sum_{\ell=1}^L \frac 12 m^2\phi_\ell^2$ and $\mathcal V^\prime_\ell(r)=\frac{\p}{\p r^\ell}\mathcal V(r_1,\dots,r_L)$. Another example of the real-analytic function $\mathcal{V}$ is the Higgs-type potential $\mathcal V(\phi)=c(|\phi|^2-m^2)^2$. Also,
\ba
\square_{g} u=
\sum_{p,q=0}^3 |g|^{-1/2}\frac  \p{\p x^p}
\left( |g|^{1/2}
g^{pq}\frac \p{\p x^q}u(x)\right),
\ea
where
$|g|=-\det(( g_{pq}(x))_{p,q=0}^3)$.

We will show that if $(M,g)$ is a solution to the Einstein field equations with scalar fields $\phi$,
and if both $(M,g)$ and $\phi$ are stationary, then $(M,g)$ is real-analytic.
{We say that $\phi$ is stationary with respect to $Z$   if $Z\phi_\ell=0$ for all $\ell=1,2,\dots,L$.
See e.g.\ \cite{Dz,Moffat} for examples on non-trivial, real-analytic, spherically symmetric solutions to equations (\ref{eq: Ein1})-(\ref{eq: Ein3}) with suitably chosen potentials $V(r)$.

M\"uller zum Hagen \cite{Muller} has shown that a stationary vacuum spacetime is real-analytic. 
This result has been generalized to several systems of Einstein equations coupled with
matter models, such as the Maxwell-Einstein equations \cite{Tod}.
Below we consider the Einstein field equations coupled with scalar fields,
and show that stationary solutions of such equations are real-analytic. Even though this result seems to be known in the folklore of the mathematical
relativity, we include for the sake of completeness a proof in Section \ref{sec_glob} as we want to
apply  Theorem   \ref{thm_global}  to this model.}

%

\begin {proposition}\label{prop: Analytic stationary}
Let $M$ be a smooth manifold, $g$ be a smooth metric tensor on $M$,
$\phi_\ell$, $\ell=1,2,\dots,L$, be smooth functions on $M$,
and suppose that $(M,g)$ and $\phi = (\phi_\ell)_{\ell=1}^L$ satisfy the Einstein field equations (\ref{eq: Ein1})-(\ref{eq: Ein3}).
Suppose, furthermore, that there is a smooth timelike Killing field $Z$ on $M$ and that $Z \phi_\ell = 0$ for $\ell = 1, 2,\dots, L$.
Then $(M,g)$ is real-analytic.
\end {proposition}


%


If the manifolds $(M,g)$ and $(\widetilde{M},\widetilde{g})$ in Theorem \ref{thm_global} are simply-connected, we have the following determination result.

\begin{corollary}
\label{cor_global 2}
{Let
$(M,g)$ and $(\tilde M,\tilde g)$ be simply connected and geodesically complete modulo scalar curvature singularities,
let $\phi$ and $\tilde \phi$ be
$\R^L$-valued scalar fields  on $M$ and $\tilde M$, respectively, and suppose
 that $(M,g,\phi)$ and $(\tilde M,\tilde g,\tilde \phi)$
satisfy the Einstein field equations with scalar fields (\ref{eq: Ein1})-(\ref{eq: Ein3}). 
Also, assume that $g$, $\phi$ and $\tilde g$, $\tilde \phi$ are stationary with respect to timelike Killing fields $Z$ and $\tilde Z$, respectively.
Moreover, let  ${\mathcal U}\subset M$  and
$\tilde {\mathcal U}\subset \tilde M$  be simply convex open sets. Let 
$\Sigma \subset {\mathcal U}$ and
$\tilde \Sigma \subset \tilde  {\mathcal U}$ be relatively compact smooth {\mltext timelike} submanifolds 
of codimension 1, and 
let  $\Psi:{\mathcal U}\to \tilde {\mathcal U}$ be a diffeomorphism such 
that $\Psi(\Sigma)=\tilde \Sigma$ {\mltext and such that
$\Psi_*$ maps the future-pointing unit normal vector field  $\nu$  of $\Sigma$
to  the future-pointing unit normal vector field $\tilde  \nu$   of $\tilde \Sigma$.} 
Assume that $\Sigma$ and $\tilde \Sigma$ are timelike convex 
near $(\hat{x_0},\hat{\xi_0}) \in T \Sigma$ and $(\Psi(\hat{x_0}),\Psi_{\hat{x_0}}(\hat{\xi_0}))\in T \widetilde \Sigma$, respectively, and
that the Lorentzian distance functions $d$ of ${\mathcal U}$ and $\tilde{d}$ of $\tilde {\mathcal U}$ satisfy
\beq\label{sigma isometry}
d(x,y)=\tilde{d}(\Psi(x),\Psi(y)), \quad\quad\quad \hbox{for all }x,y\in\Sigma.
\eeq
Then there exists an isometry  $F:(M,g)\to (\tilde M,\tilde g)$.

Furthermore, {\mltext 
assume that 
$Z$ is transversal to $\Sigma$. Also, 
suppose  $\tilde \phi=\Psi_*\phi$ on $\tilde \Sigma$,
and $\tilde Z=\Psi_*Z$ at $\Psi(\hat x_0)$ and
 $\tilde \nabla_{\Psi_* X} \tilde Z=\Psi_* \nabla_X Z$ at
 $\Psi(\hat x_0)$   for all vectors $X\in T_{\hat x_0}M$.} Here $\nabla$ and $\tilde \nabla$ are the covariant derivatives of $(M, g$) and $(\tilde M, \tilde g$), respectively. Then 
$\tilde \phi=F_*\phi$  and  $\tilde Z=F_*Z$ on $\tilde M$.}
\end{corollary}

The proof will be given in Sec.\ \ref{sec_glob}.
\medskip

{

Geodesic completeness is essential for the unique
solvability of inverse problems for partial differential equations similar
to (\ref{eq: Ein1})-(\ref{eq: Ein3}).
An important class of  
invisibility cloaking counterexamples is based on transformation optics \cite{GKLU1,GKLU2,GKLU3,GLU,Le,PSS1} and these examples are not geodesically complete.
By invisibility cloaking we mean the possibility, 
both theoretical and practical, of shielding a
region or object from
detection via electromagnetic or other physical fields.  

A model for a spacetime cloak is suggested in \cite{McCall}. There
a point $p\in \R^{1+3}$ is removed from the Minkowski space 
to obtain the spacetime $(N, g_0)$ with $N = \R^{1+3}\setminus \{p\}$ 
and
$g_0=\hbox{diag}(-1,1,1,1)$. Then $p$ is blown-up as follows:
let $M=
\R^{1+3}\setminus \overline B$, where $B$ is the Euclidean unit ball in $\R^4$, let $F:N\to M$ be a diffeomorphism, and define the metric $g_1=F_*g_0$ on $M$. The manifold 
$(M, g_1)$ can be considered as a  spacetime with
a hole $\overline B$, that can contain an object or an event, that is, a metric in $\overline B$ can be chosen freely.
When the manifolds $\overline B$ and $M$ are glued together,
we obtain a spacetime $\R^4$ with a singular metric that contains
the cloaked object in $\overline B$.
When $F$  is equal to the identity map outside a compact set, the metric
$g_1$ can be considered as ``spacetime cloaking device'' around $\overline B$.

The theory of such models have inspired laboratory experiments
\cite{Fridman} in optical systems analogous to the cloaking metric. 
The manifold $(M,g_1)$ is Ricci flat and stationary but it is not complete, or even 
 complete modulo scalar curvature singularities and therefore it does not 
 satisfy the assumptions of Theorem   \ref{thm_global} or Corollary \ref{cor_global 2}.
To the knowledge of the authors,
rigorous cloaking theory for Einstein equations, in particular the question in what
sense the non-linear Einstein field equations are valid in the cloaking examples,
is still open.

}

%
%
%

%
%
%

\subsection{Schwarzschild black hole}

\begin{figure}[htbp]
\begin{center}
\includegraphics[width=\linewidth]{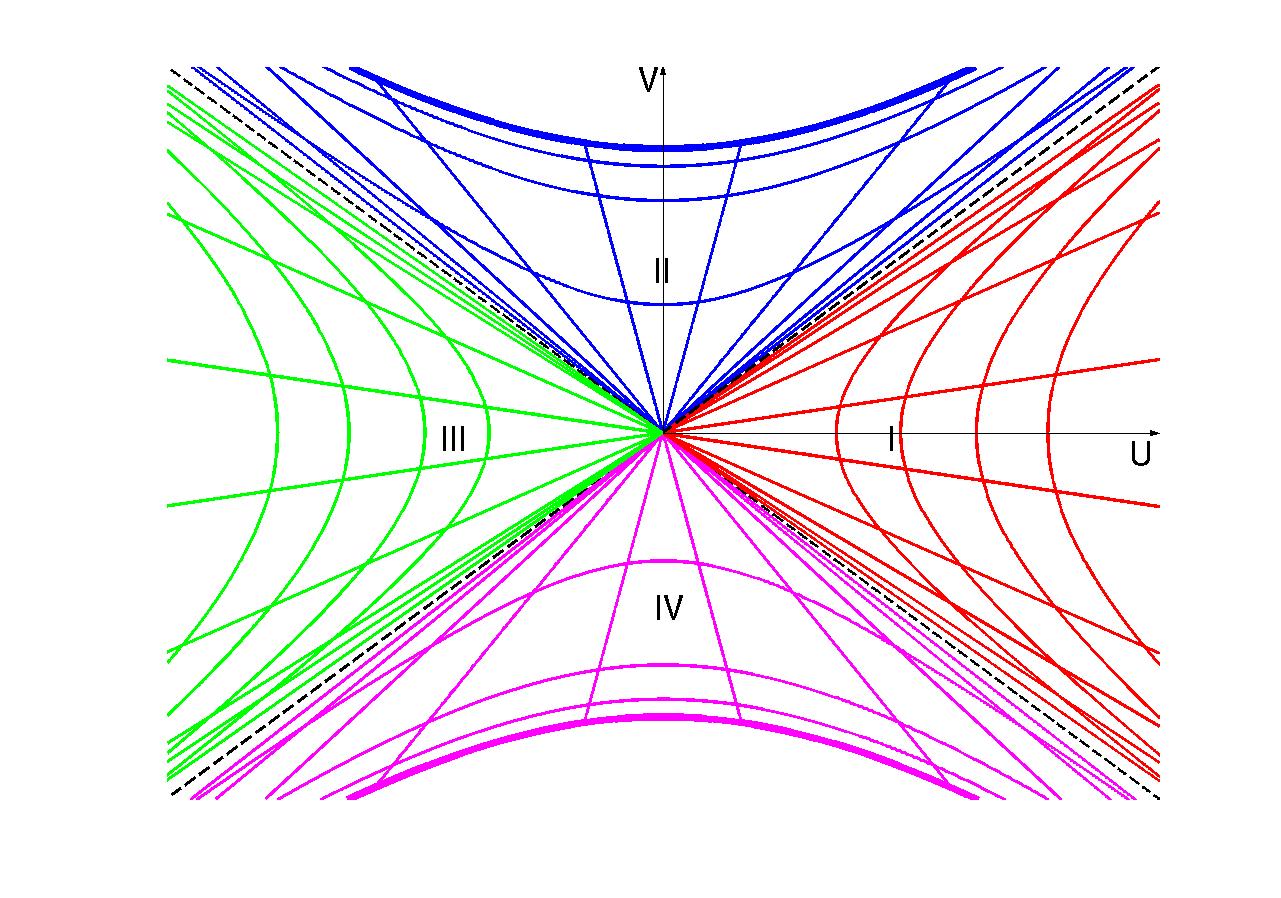}
\end{center}
\caption{{\bf Kruskal-Szekeres diagram.}  The diagram shows the domain $M$  in the $(U,V)$ plane.
 The quadrants are the black hole interior (II), the white hole interior (IV) and the two exterior regions 
 of the black hole (I and III). The  diagonal lines $U=V$ and $U=-V$, which separate these four regions, are the event horizons. The hyperbolic
 curves  which bound the top and bottom of the diagram are the physical singularities
 near which the Kretschmann scalar blow up. The hyperbolas represent contours of the Schwarzschild radial $r$ coordinate, and the lines through the origin represent contours of the Schwarzschild time coordinate $t$. The region I is usually interpreted as the exterior 
 of the black hole in one universe. The region III is  the exterior 
 of the black hole in an another, unreachable universe. 
}
\end{figure}

Let us recall the standard definition of 
 the maximally extended Schwarzschild black hole in the 
Kruskal-Szekeres  coordinates, see \cite[Sec.\ 13]{O}, \cite[Rem.\ 3.5.5] {O2}.  
Let us first consider the Schwarzschild coordinates 
 $(t,r,\theta,\phi)$, where $t\in \R$ is 
 the time coordinate (measured by a stationary clock located infinitely far from the massive body), 
 $r\in \R_+$ is the radial coordinate and $(\theta,\phi)$  are the spherical coordinates on the sphere $S^2$.
The non-extended Schwarzschild black hole is given on the chart 
$$(t,r,\theta,\phi)\in M_0=\R\times (\R_+\setminus
 \{R\})\times (-\pi/2,\pi/2)\times (-\pi,\pi)$$ by the metric
 \ba
 g = -\left(1 - \frac{R}{r} \right) dt^2 + \left(1-\frac{R}{r}\right)^{-1} dr^2 + r^2 \left(d\theta^2 + \sin^2\theta \, d\varphi^2\right), 
 \ea
 where $R=2GM$ is the Schwarzschild radius. Here $G$ is the gravitational constant and
$M$ is the Schwarzschild mass parameter, and light speed $c = 1$.
 
On the set $M_0$ the Kruskal-Szekeres  coordinates are defined
by replacing $t$ and $r$ by a new time and spatial coordinates $V$ and $U$,
\ba
  & &  V = \left(\frac{r}{R} - 1\right)^{1/2}e^{r/(2R)}\sinh\left(\frac{t}{2R}\right),\\
 & &   U = \left(\frac{r}{R} - 1\right)^{1/2}e^{r/(2R)}\cosh\left(\frac{t}{2R}\right),
 \ea
for the exterior region $r>R$, and
\ba
& &    V = \left(1 - \frac{r}{R}\right)^{1/2}e^{r/(2R)}\cosh\left(\frac{t}{2R}\right),\\
& &    U = \left(1 - \frac{r}{R}\right)^{1/2}e^{r/(sR)}\sinh\left(\frac{t}{2R}\right),
\ea
for the interior region $0<r<R$. 

In the Kruskal-Szekeres coordinates the manifold $(M_0,g)$ 
can be extended real-analytically to a larger manifold.
To consider maximal extension we denote 
\ba
M=     \{(U,V,w) \in \R^2\times S^2\ ;\
    V^2 - U^2 < 1,\ w\in S^2\}.
\ea
The metric is given by
\ba
    g= \frac{4R^3}{r}e^{-r/R}(-dV^2 + dU^2) + r^2 d\Omega^2,
\ea
 where 
$d\Omega^2{=}\ d\theta^2+\sin^2\theta\,d\varphi^2$, and
the location of the event horizon, i.e., the surface $r = 2GM$, is given by 
$V = \pm U$. 
Here $r$ is defined implicitly by the equation
$$V^2-U^2=\Big(1-\frac{r}{R}\Big)e^{r/R}.$$
Note that the metric is well defined and smooth on $M$, even at the event horizon.


The manifold $M$  with the above metric $g$ is Ricci flat, real-analytic and 
 geodesically complete modulo scalar curvature singularities.
 Indeed, the Kretschmann scalar 
for the Schwarzschild black hole
 goes to infinity as   $V^2 - U^2$  goes to $1$.

Let us consider a metric $\tilde g=g+h$, where the small perturbation $h$ is real-analytic on $M$.
We assume that also $\tilde g$ is
 geodesically complete modulo scalar curvature singularities. When $p\in M_0$, $\mathcal U\subset M$  
 is a simply convex neighborhood of a  $p$ in $(M,\tilde g)$, and $\Sigma\subset \mathcal U$
 is a 3-dimensional submanifold,
the inverse problem considered in Theorem \ref{thm_global}, can be interpreted as the question:
Do the measurements in the exterior of the event horizon on ``one side'' of a black hole 
(region I in Fig.\ 1) determine the structure of the spacetime inside the event horizon
(region  II in Fig.\ 1), or even on ``other side'' of the black hole (region  III in Fig.\ 1).
By  Theorem \ref{thm_global}, the answer to this question is positive.

Roughly speaking, this means that if the black hole spacetime has formed so
that it is real-analytic, any change of the metric in the  exterior region III changes
the metric close to the singularity, in the region II, and this  further  changes the metric
and the results of the time separation measurements in the region I.
On the other hand, the positive answer to the uniqueness of the inverse problem could be
considered as an argument that the assumption on the real-analyticity of the manifold
and the metric is too strong assumption for {\mltext physical black holes}.


\subsection{Other examples from the theory of relativity}

There are several real-analytic  solutions of vacuum Einstein equations for which
Theorems   \ref{thm_local} or  \ref{thm_global} are applicable. These include e.g.\ the maximal
analytic extension of the Kerr 
black holes, see \cite{Boyer,Visser} that are real-analytic and 
 geodesically complete modulo scalar curvature singularities.
Similarly, the problem could be considered also for
 Kerr-Newman
black holes, the
solutions corresponding to several charged black holes (with the suitably chosen masses are charges),
that is, the so-called Majumdar-Papapetrou and
Hartle-Hawking solutions \cite{Hartle}, and suitable gravitational wave solutions. 
Also, one can consider 
 cosmological models, such as
Friedmann-Lemaitre-Robertson-Walker metrics, see \cite{SW}.
The detailed analysis
of the inverse problem for
these manifolds are outside the scope of this paper.

\subsection{Material particles and clocks}

In this example we construct a specific timelike hypersurface $\Sigma$
satisfying the hypothesis {\bf H}, and give a physical motivation behind this construction.
We recall some facts from the theory of relativity, our main references are \cite{O} and \cite{SW}.
A point $(p,u)$ on the tangent bundle $TM$
is called an instantaneous observer if $u \in T_p M$ is a future-pointing timelike unit vector. Given such an instataneous observer, the Lorentzian vector space $T_{p}M$ admits a direct sum decomposition as
$$T_{p}M=\mathbb{R}u\oplus u^{\perp},$$
where $\mathbb{R}u$ is the 1-dimensional subspace spanned by $u$, and $u^{\perp}$ is an $n$-dimensional spacelike subspace which is orthogonal to $\mathbb{R}u$. $\mathbb{R}u$ is called the observer's time axis and $u^{\perp}$ the observer's restspace.

A smooth curve $\alpha:I\rightarrow M$ is called a material particle if it is future-pointing timelike and $(\dot{\alpha}(\tau),\dot{\alpha}(\tau))_{g}=-1$ for all $\tau\in I$.
A material particle is said to be freely falling if it is a (necessarily timelike) geodesic.
We recall that if a material particle $\alpha$ passes through the point $p$, say $\alpha(0) = p$,
then the energy $E$ and momentum $q$ of $\alpha$ as measured by the observer $(p,u)$ are
\begin{align*}
E = -(\dot{\alpha}(0), u)_{g}, \quad q = \dot{\alpha}(0) - E u.
\end{align*}
Thus we have the decomposition $\dot\alpha(0)=Eu+q$. Moreover, the Newtonian velocity as measured by $(p,u)$ is $v = q / E \in u^{\perp}$, see \cite[p. 45]{SW}.

Let $\mathcal{U}$ be a simply convex neighborhood of $p$ in $M$.
We recall that $d:\mathcal{U}\times\mathcal{U}\rightarrow\mathbb{R}$ is defined in (\ref{def_timesep})
and that $\gamma_{x,y}$ is the unique geodesic in $\mathcal{U}$
connecting points $x$ and $y$ in $\mathcal U$.
Physically, if $\gamma_{x,y}$ is future-pointing and parametrized by arc length, one may think of $\gamma_{x,y}$ as a freely falling material particle, then $d(x,y)$ gives the elapsed proper time of the particle from the event $x$ to the event $y$.

Let $c_{0}\in (0,1)$ be a constant. We define $\mathcal{G}$ to be the set of freely falling material particles $\alpha$ with $\alpha(0) = p$ and with the Newtonian velocities satisfying $|v|_{g} = c_0$. Choose $\epsilon>0$ to be small so that the topological closure of the set
$$\Sigma_{\epsilon}:=\{\alpha(\tau):\alpha\in\mathcal{G}, \tau\in(0,\epsilon)\}$$
is contained in the simply convex neighborhood $\mathcal{U}$.

Notice that if $\alpha\in\mathcal{G}$, then by combining $(\dot{\alpha}(0),\dot{\alpha}(0))_{g}=-1$ and $|v|_{g}=|q/E|_{g}=c_{0}$ we get
$$
\begin{array}{rl}
-1= & (\dot{\alpha}(0),\dot{\alpha}(0))_{g}=(Eu+q,Eu+q)_{g}=-E^{2}+E^{2}(q/E,q/E)_{g}\vspace{1ex}\\
= & -E^{2}+c^{2}_{0}E^{2}.\\
\end{array}
$$
Hence $E=(1-c^{2}_{0})^{-1/2}$ and $|q|_{g}=c_{0}E$.
Therefore, with the exponential map $\exp_{p}$, we can write $\alpha(\tau)$ as
$$\alpha(\tau)=\exp_{p}(\tau\dot{\alpha}(0))=\exp_{p}\left(\displaystyle\frac{\tau}{\sqrt{1-c^{2}_{0}}}(u+c_{0}\xi)\right) \quad \textrm{ for some  }\xi\in S^{n-1}$$
where $S^{n-1}$ denotes the unit sphere in the rest space $u^{\perp}$ of the instantaneous observer $(p,u)$. This expression then provides a parametrization of $\Sigma_{\epsilon}$. In fact, the map
$$
\begin{array}{rl}\vspace{1ex}
\kappa:(0,\epsilon)\times S^{n-1}\rightarrow & \mathcal{U}  \\ \vspace{1ex}
(\tau,\xi)\mapsto & \exp_{p}\left(\displaystyle\frac{\tau}{\sqrt{1-c^{2}_{0}}}(u+c_{0}\xi)\right)
\end{array}
$$
is a diffeomorphism of $(0,\epsilon)\times S^{n-1}$ onto $\Sigma_{\epsilon}$.

Let $(\tilde M, \tilde g)$ be another $(1+n)$-dimensional smooth Lorentzian manifold, and $(\tilde{p},\tilde{u})$ another instantaneous observer with $\tilde{p}\in M$ and $\tilde{u}\in T_{\tilde{p}}\widetilde{M}$. Likewise, we can define the set $\tilde \Sigma_{\epsilon}$ and assume $\epsilon >0$ is so small that $\tilde \Sigma_{\epsilon}$ is contained in a simply convex neighborhood $\widetilde{\mathcal{U}}$ of $\tilde{p}$. We also have that
$$
\begin{array}{rl}\vspace{1ex}
\tilde \kappa:(0,\epsilon)\times S^{n-1}\rightarrow & \tilde{\mathcal{U}} \\ \vspace{1ex}
(\tau,\xi)\mapsto & \widetilde{\exp}_{\tilde p}\left(\displaystyle\frac{\tau}{\sqrt{1-c^{2}_{0}}}(\tilde u+c_{0}\xi)\right)
\end{array}
$$
is a diffeomorphism of $(0,\epsilon)\times S^{n-1}$ onto $\tilde{\Sigma}_{\epsilon}$. Define the Lorentzian distance function $\tilde d$ on $\widetilde{\mathcal{U}}$ analogously to $d$. We can identify $\Sigma_{\epsilon}$ with $\widetilde{\Sigma}_{\epsilon}$ via the diffeomorphism $\tilde{\kappa}\circ\kappa^{-1}$. Using the notation of Theorem \ref{thm_local}, we can take $\Phi = \tilde{\kappa}\circ\kappa^{-1}$.

Now we relate the quantities appearing in Theorem \ref{thm_local} with the quantities in the present example.
Clearly the closure of $\Sigma_\epsilon$ is compact in $\mathcal{U}$.
As $\Sigma_{\epsilon}$ consists of future-pointing timelike geodesics parametrized by $S^{n-1}$, it is a timelike smooth submanifold of codimension $1$. To show that the hypothesis \textbf{H} holds, simply notice that $\Sigma_{\epsilon}$ is diffeomorphic to a cone minus the tip in $T_{p}M$, under the exponential map $\exp_{p}$. Since the cone minus the tip in $T_{p}M$ satisfies \textbf{H}, so does its image $\Sigma_{\epsilon}$.

Let us now give a physical interpretation of the above example. Imagine that a primary observer $(p,u)$ shoots out numerous other observers with the same Newtonian speed $c_{0}$ in all the directions, and these secondary observers move under the influence of gravitation. Using the notation as above, each secondary observer can be denoted by a freely falling (i.e., under the influence of gravitation only) material partical $\alpha$ with $\alpha(0)=p$, thus the collection of the secondary observers is the set $\mathcal{G}$ and the trajectories of the secondary observers in the spacetime $\mathbb{R}^{1+3}$ form the submanifold $\Sigma_\epsilon$, provided $\epsilon>0$ is small enough. Suppose each secondary observer carries many clocks, one of them is kept to read his own elapsed time, all the other clocks are to be emitted. Suppose when these secondary observers are shoot out from the primary observer, they start emitting continuously to each other the clocks which also move only under the influence of gravitation. Each emitted clock then records the elapsed time between the two secondary observers: the launcher and the receiver.

When a clock hits the receiver, he/she can read the clock to find out the elapsed time of this clock. After transmitting all the elapsed times collected this way to the primary observer, the primary observer standing at $p$ then knowns the elapsed time between any two secondary observers, which amounts to knowing the Lorentzian distance function $d$ on $\Sigma_\epsilon$. By the analysis in this section and Theorem \ref{thm_local}, the primary observer is able to determine (the $C^\infty$-jet of) the metric structure of the universe at $p$ at the beginning of the experiment. Furthermore, if the metric structure satisfies the assumptions of Theorem \ref{thm_global}, this measurement determines the universal Lorentzian covering space of the universe.



\subsection{Incompatible extensions}
\label{ex_incompatible}

Let us now show that the assumption that the manifold is
geodesically complete modulo scalar curvature singularities is essential in
Theorem \ref{thm_global} and that the universal covering space can not be determined without
some kind of completeness assumption.
We do this by constructing a counterexample of
two manifolds one of which is not geodesically complete modulo scalar curvature singularities,
and such that both the manifolds have the same time measurement data.

We recall, see \cite[Def. 6.15]{Beem} and \cite[p.\ 58]{HE}, that an extension of a
Lorentzian manifold $(M,g)$ is a Lorentzian manifold $(M',g')$ together with a
map $f:M\to M'$ onto a proper open subset $f(M)$ of $M'$ such that $f:M\to f(M)$ is a diffeomorphism, and
$f_*g=g'$.
Also,
if $(M,g)$ has no  extension, it is said to be inextendible or maximal Lorentzian manifold.
If $(M,g)$, $(M',g')$ and the map $f$ are real-analytic, we say that $(M',g')$
is a real-analytic extension of $(M,g)$.

Let us consider the product manifold $\R \times S^2$ endowed with
the Lorentzian metric $\tilde g=-dt^2+\hat g_{S^2}$, where $\hat g_{S^2}$  is the standard Riemannian metric of
$S^2$. Let
$p_1$ and $p_2$
be the South and the North pole.
Also, let $M=\R \times (S^2 \setminus I(p_1, p_2))$ where $I(p_1, p_2)$ is one of the shortest
Riemannian geodesics connecting $p_1$ to $p_2$, that is, an arc of a great circle connecting
the South pole to the North pole. Here, the arc $I(p_1, p_2)$ is closed and contains the points
$p_1$   and $p_2$.
 Let us endow $M$ with the metric $g_M=\tilde g|_M$.

Next we construct two real-analytic extensions for manifold $M$, denoted by $M_e$ and $M_c$.
First, let
$M_e=\R \times S^2$  be a Lorentzian manifold with metric $g_e=\tilde g$.
Observe that manifold $M_e$ is simply connected and geodesically complete.
Second, let $N=\R \times (S^2 \setminus \{p_1, p_2\})$
 be a Lorentzian manifold with the metric $g=\tilde g|_N$ and
let $(M_c,g_c)$ be the universal covering space of $(N,g)$.
Using the spherical coordinates, we see that
$N$  is homeomorphic to $\R\times (0,\pi)\times S^1$ and the manifold $M_c$  is homeomorphic to
the simply connected manifold $\R\times (0,\pi)\times \R$.
The obtained manifolds $M_e$ and $M_c$ are real-analytic extensions of the manifold $M$.

Let $p\in S^2 \setminus I(p_1, p_2)$ and $B_{S^2}(p,r)$ denote
the open ball of radius $r$ and center $p$ in the Riemannian manifold $S^2$.
When $s>0$ is small enough, let $\mathcal U=(-s,s)\times B_{S^2}(p,s)$.
Then the manifolds $M_e$ and $M_c$ contain the set  $\mathcal U \subset M$ that is
a simply convex neighborhood of $x_0=(0,p)$. More precisely, both $M_e$ and $M_c$
contain a simply convex open set that is isometric to $(\mathcal U,\tilde g|_{\mathcal U})$
that can be considered as the data given in Theorems \ref{thm_local} and \ref{thm_global}.
 Note that
 the manifolds $M_e$ and $M_c$ are simply connected and therefore the both
 are their own universal covering spaces.

Let us next show that there are no
 3-dimensional Lorentzian manifold $(M_1,g_1)$ such that both
 $M_e$ and $M_c$ could be isometricly embedded in $M_1$.
 To show this, let us assume the opposite, that such manifold $M_1$ exists.
 Since $M_e$ is complete, it follows from \cite[Prop.\ 6.16]{Beem} that it is inextendible, i.e.  maximal.
Thus $M_1$ is isometric
 to $M_e$ and there  is an isometric embedding $F:M_c\to M_e$.

 Let us recall, see \cite[Sec.\ 1A]{O}, that when $\pi:\R\times S^2\to \R$ and
 $\sigma:\R\times S^2\to S^2$ are operators $\pi(s,q)=s$ and $\sigma(s,q)=q$,
 then the lift $L_\sigma X$, of a vector field $X$ defined on $S^2$, is the unique
 vector field $\tilde X=L_\sigma X$ defined on $\R\times S^2$   such
 that $d\sigma(\tilde X)=X$ and $d\pi(\tilde X)=0$. Then, at the point $(s,q)$,   we have  $\tilde X|_{(s,q)}\in T_{(s,q)} (\{s\}\times S^2)$.

 Using \cite[Prop.\ 58 on page 89]{O}, we see that the curvature
operator $R$ of $\R \times S^2$ at $x=(s,q)\in \R \times S^2$
is such that for all $X,Y\in T_x(\R \times S^2)$,
\ba
R(\p_s,X)Y=0\quad\hbox{and}\quad R(X,Y)\p_s=0.
\ea
Moreover, for all $U,V,W\in T_{(s,q)}(\{s\}\times S^2)$,  we have at $(s,q)\in  \R\times S^2$
\ba
R(U,V)W=L_\sigma(R_{S^2}(\sigma_*U,\sigma_*V)\sigma_*W),
\ea
that is, $R(U,V)W$   is equal to the lift of
 $R_{S^2}(\sigma_*U,\sigma_*V)\sigma_*W$,
 where $R_{S^2}$ is the Riemannian curvature operator of $S^2$.
This implies that the linear space
$$\mathcal V_x=\{R(X,Y)Z\in T_{(s,q)} (\R \times S^2)\ |\ X,Y,Z\in T_{(s,q)} (\R \times S^2)\},\quad x=(s,q),$$
is equal to the space $T_{(s,q)}(\{s\}\times S^2)$.

Since $F:M_c\to M_e$
is an isometry, we see that at $x\in M_c$ and $y=F(x)\in M_e$
the differential of $F$, $dF:T_xM_c\to T_yM_e$, maps $\mathcal V_x$ onto $\mathcal V_y$.
Let now $x=(s,q)\in M\subset N$ and $\xi \in  \mathcal V_x\subset T_xM$ and
so that the geodesic $\gamma^{M_e}_{x,\xi}$ on $M_e$  is a great circle on $\{s\}\times S^2$.
Moreover, let us assume that $\xi$
is such that $\gamma^{M_e}_{x,\xi}$  does not intersect $\R\times \{p_1\}$ or
$\R\times \{p_2\}$. 
Then, considering $(x,\xi)\in TN$ as an element of $TM_c$ we see that
the geodesic $\gamma^{M_c}_{x,\xi}$ on $M_c$ is complete  and it is  homeomorphic to
the real axis. However, its image of $F$, that is, $F(\gamma^{M_c}_{x,\xi}(\R))=\gamma^{M_e}_{F(x),F_{\ast}(\xi)}(\R)\subset M_e$ is
a closed geodesic that is homeomorphic to $S^1$. This is in contradiction with the assumption that
$F$ is an isometric embedding.

Summarizing, we have seen that the manifold $M$ has two real-analytic extensions
$M_e$ and $M_c$ that can not be isometricly embedded in any  connected manifold $M_1$ of dimension 3 that
would contain  both of them and both these manifolds contain a subset $\mathcal U$ with metric $\tilde g|_{\mathcal U}$. This shows that assumption that the manifold is geodesically complete modulo scalar curvature singularities is essential in
Theorem \ref{thm_global}.

\section{Proof on the $C^{\infty}$-jet determination}
\label{sec_local}

{In this section we investigate the local inverse problem of the  $C^\infty$-jet determination.} First, we establish some basic facts about the Lorentzian distance function $d(x,y)$. Sometimes we would like to fix $x$ and think of $d(x,y)$ as a function of $y$, in this case we may write $d(x,y)$ as $d^{+}_{x}(y)$; sometimes we would like to fix $y$ and think of $d(x,y)$ as a function of $x$, in this case we may write $d(x,y)$ as $d^{-}_{y}(x)$. Notice that $d^{+}_{x}(y)>0$ if and only if $y\in I^{+}_{\mathcal{U}}(x)$; $d^{-}_{y}(x)>0$ if and only if $x\in I^{-}_{\mathcal{U}}(y)$.
{We start with the following simple results whose proof we include
for the convenience of the reader.}

\begin{lemma} \label{measurement}
Let  $\mathcal  U$ be a simply convex neighborhood on a smooth Lorentzian manifold $(M,g)$.
Then
\begin{enumerate}[(i).]
    \item $d$ is continuous in $\mathcal{U}\times\mathcal{U}$.
    \item $d$ is smooth in $\{(x,y)\in\mathcal{U}\times\mathcal{U}:x\ll y\}$.
    \item Let $c:[0,\epsilon)\rightarrow\mathcal U$ be a smooth curve with $c(0)=x$ and $x\ll c(s)$ for all $s\in (0,\epsilon)$, then
        $$\displaystyle\lim_{s\rightarrow 0+}\frac{d^{+}_{x}(c(s))}{s}=\lim_{s\rightarrow 0+}\frac{d(x,c(s))}{s}=|\dot{c}(0)|_{g}.$$
    \item For $y\in I^{+}_{\mathcal{U}}(x)$, let $\gamma:[0,\ell]\rightarrow \mathcal{U}$ be a future pointing timelike radial geodesic with $\gamma(0)=x$, $\gamma(\ell)=y$, then
        $$\textrm{grad}\;d^{+}_{x}(y)=-\frac{\dot{\gamma}(\ell)}{|\dot{\gamma}(\ell)|_{g}}.$$
    \item The eikonal equation
    $$(\textrm{grad}\;d^{+}_{x}(y),\textrm{grad}\;d^{+}_{x}(y))_{g}=-1$$
    holds for $y\in I^{+}_{\mathcal{U}}(x)$.
\end{enumerate}
\end{lemma}
\begin{proof}
\begin{enumerate}[(i).]
\item  As the exponential map is a radial isometry \cite[Chapter 5 Lemma 13]{O}, we have
$$d(x,y)=|\exp^{-1}_{x}y|_{g}=\sqrt{|\left(\exp^{-1}_{x}y,\exp^{-1}_{x}y\right)_{g}|}.$$
It remains to show that $\exp^{-1}_{x}y$, as a function of $(x,y)$, is continuous. In fact we will prove a stronger result: $\exp^{-1}_{x}y$ is a smooth function of $(x,y)\in\mathcal{U}\times\mathcal{U}$.

To this end, introduce the notations
$$
\begin{array}{rl}
\mathcal{D}:= & \{(x,v)\in TM:(x,\exp_{x}v)\in\mathcal{U}\times\mathcal{U}\},\vspace{1ex}\\
\mathcal{D}_{x}:= & \{v\in T_{x}M:(x,v)\in\mathcal{D}\}. \\
\end{array}
$$
It is easy to see that $\mathcal{D}$ and $\mathcal{D}_{x}$ are open subsets of $TM$ and $T_{x}M$, respectively. Consider the map
$$E:\mathcal{D}\rightarrow \mathcal{U}\times\mathcal{U}, \quad\quad E(x,v)=(x,\exp_{x}v).$$
By simply convexity of $\mathcal{U}$, the exponential map $\exp_{x}:\mathcal{D}_{x}\rightarrow\mathcal{U}$ is non-singular at any $v\in\mathcal{D}_{x}$, thus by \cite[Chapter 5 Lemma 6]{O}, $E$ is non-singular at any $(x,v)\in\mathcal{D}$. Notice that $E$ is a smooth map between manifolds of the same dimension, we conclude $E$ is a local diffeomorphism. Notice further that $E$ is bijective on $\mathcal{D}$, we come to the conclusion that $E$ is in fact a diffeomorphism. By setting $y=\exp_{x}v$ it follows that $\exp^{-1}_{x}y$ is smooth for $(x,y)\in\mathcal{U}\times\mathcal{U}$. This completes the proof.

\item If $x\ll y$, then the smooth function $\exp^{-1}_{x}y$ is non-vanishing. From the expression
        \beq\label{inverse exp}
       d(x,y)=|\exp^{-1}_{x}y|_{g}=\sqrt{-\left(\exp^{-1}_{x}y,\exp^{-1}_{x}y\right)_{g}}\eeq
        it is straightforward that $d$ is smooth in $\{(x,y)\in\mathcal{U}\times\mathcal{U}:x\ll y\}$.

\item  Choose a normal coordinate chart $(\varphi,\xi^{i})$ where $\varphi=\exp^{-1}_{x}$. Let $\xi(s)=\varphi(c(s))$. The function $d^{+}_{x}$ composed with this normal coordinate chart reads
$$d^{+}_{x}\circ\varphi^{-1}:\xi(s)\mapsto\displaystyle\sqrt{-(\xi(s),\xi(s))_{g}}.$$
Notice that $\xi(0)=\varphi(x)=0$, so $\xi(s)=\dot{\xi}(0)s+\mathcal{O}(s^{2})$. From this we conclude
$$\displaystyle\lim_{s\rightarrow 0+}\frac{d^{+}_{x}(c(s))}{s}=\lim_{s\rightarrow 0+}\frac{d^{+}_{x}\circ\varphi^{-1}(\xi(s))}{s}=|\dot{\xi}(0)|_{g}=|\dot{c}(0)|_{g}.$$\\

\item  The exponential map is a radial isometry, thus
$$d^{+}_{x}(\gamma(t))=d(x,\gamma(t))=|\dot{\gamma}(0)|_{g}t.$$
Differentiate to get
\begin{equation}\label{chain_rule}
(\textrm{grad}\;d^{+}_{x}(\gamma(t)),\dot{\gamma}(t))_{g}=|\dot{\gamma}(0)|_{g}.
\end{equation}
Since $\dot{\gamma}(t)$ is orthogonal to the level sets of $d^{+}_{x}$ by Gauss Lemma \cite[Chapter 5 Lemma 1]{O}, and since $\textrm{grad}\,d^{+}_{x}$ is also orthogonal to the level sets of $d^{+}_{x}$, there exists a function $C(t)$ such that $\textrm{grad}\;d^{+}_{x}(\gamma(t))=C(t)\dot{\gamma}(t)$. As $\gamma$ is a geodesic, $|\dot{\gamma}(0)|_{g}=|\dot{\gamma}(t)|_{g}$ for all $0\leq t\leq\ell$. Therefore, using \eqref{chain_rule} we derive that
$$C(t)=-\frac{1}{|\dot{\gamma}(t)|_{g}}.$$
Setting $t=\ell$ completes the proof.

\item  Using the result in (iv) and that $(\dot{\gamma}(t),\dot{\gamma}(t))_{g}=-|\dot{\gamma}(t)|^{2}_{g}$ we have
$$(\textrm{grad}\;d^{+}_{x}(y),\textrm{grad}\;d^{+}_{x}(y))_{g}=(-\frac{\dot{\gamma}(t)}{|\dot{\gamma}(t)|_{g}},-\frac{\dot{\gamma}(t)}{|\dot{\gamma}(t)|_{g}})_{g}=-1.$$
\end{enumerate}
\end{proof}

\textbf{Remark:} Throughout this paper we only assume to know the Lorentzian distance between any two points on $\Sigma$. However, Lemma \ref{measurement}(i) says that by continuity we can further know the Lorentzian distance between any two points on the closure $\overline{\Sigma}$. In other words, we know not only $d|_{\Sigma\times\Sigma}$, but also $d|_{\overline{\Sigma}\times\overline{\Sigma}}$. This observation is useful in certain circumstances.
\vspace{2ex}


Without loss of generality we can suppose in the assumption of Theorem \ref{thm_local} that $\hat{\xi_0}$ is a past-pointing timelike vector, and we do this assumption below. From now on, we will systematically use $\sim$ to denote the quantities which are related via the diffeomorphism $\Phi$; for instance, $\tilde{x}:=\Phi(x)$.

As the first step towards proving Theorem \ref{thm_local}, the following proposition says that the restriction of the Lorentzian distance function $d|_{\Sigma\times\Sigma}$ determines the metric $g$ on the tangent bundle of $\Sigma$.

\begin{proposition}\label{boundary}
Under the assumption of Theorem 1, for any $x\in\Sigma$ and any $\xi\in T_{x}\Sigma$, we have
$$(\xi,\xi)_{g}=(\tilde{\xi},\tilde{\xi})_{\tilde{g}}$$
where $\tilde{\xi}:=\Phi_{\ast x}(\xi)$ is the image of $\xi$ under the push-forward $\Phi_{\ast}$ at $x$. Consequently, by polarization $g=\Phi^{\ast}\tilde{g}$ on $T_{x}\Sigma$.
\end{proposition}

\begin{proof}

For any fixed $x\in\Sigma$, since $\Sigma$ is a timelike submanifold, we can find a future-pointing timelike vector $\xi_{0}\in T_{x}\Sigma$. Let $c:[0,\epsilon)\rightarrow\Sigma\cap I^{+}_{\mathcal{U}}(x)$ be a smooth curve with $c(0)=x$ and $\dot{c}(0)=\xi_{0}$. By the assumption of Theorem \ref{thm_local}, $\tilde{d}(\tilde{x},\tilde{c}(s))=d(x,c(s))>0$, we conclude $\tilde{c}(s)\in\tilde{I}^{+}_{\widetilde{\mathcal{U}}}(\tilde{x})$. Hence $\tilde{c}:[0,\epsilon)\rightarrow\widetilde{\Sigma}\cap\tilde{I}^{+}_{\widetilde{\mathcal{U}}}(\tilde{x})$ is a smooth curve with $\tilde{c}(0)=\tilde{x}$ and $\dot{\tilde{c}}(0)=\tilde{\xi}_{0}$. By Lemma \ref{measurement}(iii)
$$-(\tilde{\xi}_{0},\tilde{\xi}_{0})_{\tilde{g}}=\Big(\displaystyle\lim_{s\rightarrow 0+}\frac{\tilde{d}^{+}_{\tilde{x}}(\tilde{c}(s))}{s}\Big)^{2}=\Big(\displaystyle\lim_{s\rightarrow 0+}\frac{d^{+}_{x}(c(s))}{s}\Big)^{2}=-(\xi_{0},\xi_{0})_{g}.$$
Since this identity is true for all future-pointing timelike vectors in $T_{x}\Sigma$, we conclude that the two quadratic forms $(\Phi_\ast \cdot, \Phi_\ast \cdot)_{\tilde{g}}$ and $(\cdot,\cdot)_g$ coincide on the open set of timelike vectors, hence they must be equal everywhere.
\end{proof}

We introduce a local coordinate system which is an analogue of the semi-geodesic coordinates in Riemannian geometry. As the hypersurface $\Sigma$ is an $n$-dimensional manifold, near the fixed point $\hat{x_0}\in\Sigma$ we can find a coordinate chart $(W,(x^{1},\dots,x^{n}))$ such that $W$ is a neighborhood of $\hat{x_0}$ in $\Sigma$ and the closure $\overline{W}$ is compact in $\Sigma$. Let $\nu$ be the unit normal vector field on $\Sigma$, chosen as in the hypothesis \textbf{H}. For small $\delta>0$, we can define a diffeomorphism using $W$ as follows:
\begin{equation}\label{Psi}
\Psi(x,r):=\exp_{x}(r\nu(x)) \quad\quad (x,r)\in W\times (-\delta,\delta).
\end{equation}
Geometrically $\Psi^{-1}$ parameterizes a tubular neighborhood of $W$ in $\mathcal{U}$. Similarly we can define
$$\tilde{\Psi}(\tilde{x},r):=\exp_{\tilde{x}}(r\tilde{\nu}(\tilde{x})) \quad\quad (\tilde{x},r)\in\widetilde{W}\times (-\delta,\delta)$$
where $\tilde{\nu}$ is the normal vector field to $\widetilde{\Sigma}$ chosen as in the hypothesis \textbf{H}, $\tilde{x}=\Phi(x)$, and $\widetilde{W}:=\Phi(W)$ is the image of $W$ under the diffeomorphism $\Phi$.
Let $id:(-\delta,\delta)\rightarrow (-\delta,\delta)$ be the identity map. By identifying $x\in W$ with $(x,0)\in W\times(-\delta, \delta)$, the diffeomorphism
\begin{equation}\label{diffeo}
\widetilde{\Psi}\circ(\Phi\times id)\circ\Psi^{-1}:W\times(-\delta,\delta)\rightarrow \tilde{W}\times(-\delta,\delta),
\end{equation}
is precisely $\Phi|_{W}$ when restricted to $W\times\{0\}$. In other words, the map \eqref{diffeo} extends $\Phi|_{W}$ to a diffeomorphism which identifies the tubular neighborhood $\Psi(W\times (-\delta,\delta))$ with the tubular neighborhood $\tilde{\Psi}(\widetilde{W}\times (-\delta,\delta))$. We will continue using a $\sim$ to indicate that the quantities are related via this diffeomorphism.

Using the coordinates on $W$, $(x^{1},\dots,x^{n},r)$ constitute coordinates in the tubular neighborhood $\Psi(W\times(-\delta,\delta))$. Similarly $(\tilde{x}^{1},\dots,\tilde{x}^{n},r)$ form local coordinates for the tubular neighborhood $\tilde{\Psi}(\widetilde{W}\times(-\delta,\delta))$. In these coordinates, the metrics $g$ and $\tilde{g}$ can be expressed as
\begin{equation}\label{metric}
\begin{array}{rl}\vspace{1ex}
g= & \displaystyle\sum^{n}_{i,j=1}g_{ij}(x,r)dx^{i}dx^{j}+dr^{2},\\ \vspace{1ex}
\tilde{g}= & \displaystyle\sum^{n}_{i,j=1}\tilde{g}_{ij}(\tilde{x},r)d\tilde{x}^{i}d\tilde{x}^{j}+dr^{2}.\\
\end{array}
\end{equation}

Now we are ready to prove our first main theorem. Roughly speaking, we shoot some timelike geodesics from near $\hat{x}_0$, which by the timelike convexity assumption will intersect $\overline{\Sigma}$. For those long geodesics, we adapt the proof of \cite[Theorem 1]{SU} which is in the context of Riemannian geometry. For the short geodesics, we follow the idea of the proof of \cite[Theorem 2.1]{LSU} in the Riemannian setting to write distances as integrals over geodesics.

\medskip

\begin{proof}[Proof of Theorem \ref{thm_local}]
Using the coordinates in \eqref{metric}, we only need to determine $C^{\infty}$-jet of each component $g_{ij}$ at $\hat{x_0}$. As a conclusion of Proposition \ref{boundary}, the functions $g_{ij}$ are uniquely determined on $\{(x,r):r=0\}$, from this we can find all tangential derivatives of $g_{ij}$ at $\hat{x_0}$. Next we will show that $d|_{\Sigma\times\Sigma}$ uniquely determines the normal derivatives $\frac{\partial^{k}g_{ij}}{\partial r^{k}}(\hat{x_0})$ for $k=1,2,\dots$; that is, we will show
\begin{equation}\label{normal}
\displaystyle\frac{\partial^{k} g_{ij}}{\partial r^{k}}(\hat{x_0})=\frac{\partial^{k}\tilde{g}_{ij}}{\partial r^{k}}(\Phi(\hat{x_0})) \quad\quad\quad  k=1,2,\dots.
\end{equation}
We remark that in the following proof of \eqref{normal}, the key information of $g$ that is used is the knowledge of its tangential derivatives on $\{(x,r): r=0\}$, thus the proof is also valid if $g_{ij}$ in \eqref{normal} is replaced by any of its tangential derivative $\frac{\partial^{s_1+ \dots +s_n} g_{ij}}{(\partial x^1)^{s_1} \dots (\partial x^n)^{s_n}}$, and in this way we can determine any mixed derivative of the form $\frac{\partial^{s_1+ \dots +s_n+s_{n+1}} g_{ij}}{(\partial x^1)^{s_1} \dots (\partial x^n)^{s_n} (\partial r)^{s_{n+1}}}$.

In order to make the proof of \eqref{normal} clear, we divide it into two steps.

\emph{Step 1:} Let us start with the case when $k=1$. We will employ two types of argument alternately: one is constructive, that is, we give explicit procedures on how to recover quantities related to the metric $g$ from the measurement function $d|_{\Sigma\times\Sigma}$; the other is non-constructive: we show that some quantities, which are related to $g$ and $\tilde{g}$ respectively, are identical under the assumption that $d(x,y)=\tilde{d}(\tilde{x},\tilde{y})$ for all $x,y\in\Sigma$.

For the metric $g$, let $(x_0,\xi_0)\in\hat{U}$ with $\xi_0$ a past-pointing timelike unit vector, and let $\nu$ be the unit normal vector field as in the hypothesis \textbf{H}. Define a sequence of vectors $\xi_{l}\in T_{x_{0}}\Sigma$ by $$\xi_{l}:=\xi_{0}+\frac{1}{l}\nu(x_{0})=\displaystyle\sum^{n}_{i=1}\xi^{i}_{0}\left.\frac{\partial}{\partial x^{i}}\right|_{x_{0}}+\frac{1}{l}\left.\frac{\partial}{\partial r}\right|_{x_{0}}.$$
Here the positive integer $l$ is chosen to be sufficiently large so that $\xi_{l}$ is also a timelike past-pointing vector. Let $\gamma_{l}$ be the unique geodesic issued from $x_{0}$ in the direction $\xi_{l}$. By the hypothesis \textbf{H}, $\gamma_{l}$ will intersect $\overline \Sigma$
for some $t_{l}>0$, without loss of generality we may assume that $t_{l}$ is the smallest parameter value such that the intersection happens. The corresponding encounter point $y_{l}:=\exp_{x_{0}}(t_{l}\xi_{l})$ will be different from $x_{0}$ since $\mathcal{U}$ is simply convex (see pictures below). As $\overline{\Sigma}$ is compact, $\{t_{l}\}$ is a bounded sequence, hence has a convergent subsequence which we assume to be itself and write $t_{l}\rightarrow t_{0}\in\mathbb{R}$ and $y_{l}\rightarrow y_{0}\in\overline{\Sigma}$ as $l\rightarrow\infty$. Based on $\{y_{l}\}$ and $y_{0}$, we can define functions $d^{+}_{y_{l}},l=1,2,\dots$ and $d^{+}_{y_{0}}$. Notice that although $y_{l}\neq x_{0}$ for each $l$, it is possible that $y_{0}=x_{0}$. In the following we consider two cases: $y_{0}\neq x_{0}$ and $y_{0}=x_{0}$. For the case $y_{0}\neq x_{0}$ we employ a constructive agrument, while for the case $y_{0}=x_{0}$ we prove the uniqueness in a non-constructive way. We write them as two claims.\\

\emph{Claim 1: If $y_{0}\neq x_{0}$, we can uniquely determine $\sum^{n}_{i,j=1}\frac{\partial g^{ij}}{\partial r}(x_{0})\xi^{i}_{0}\xi^{j}_{0}$.}\vspace{2ex}

\begin{figure}[h]
\begin{center}
\includegraphics[height=0.2\textheight]{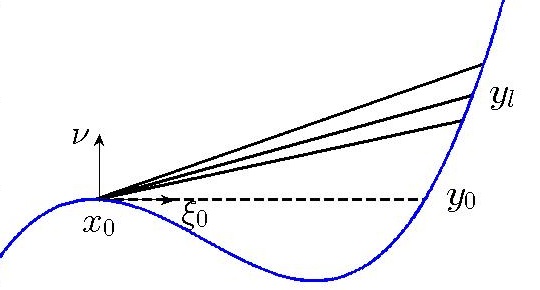}
\end{center}
\caption{The case where $y_0\neq x_0$.} 
\end{figure}

To prove Claim 1, first notice that in this case $y_{0}=\exp_{x_{0}}(t_{0}\xi_{0})\in I^{-}_{\mathcal{U}}(x_{0})$, which means $y_0\ll x_0$. By Lemma \ref{measurement}(ii), $d^{+}_{y_{l}}$ is smooth in a neighborhood, say $U_l\subset\mathcal{U}$, of $x_0$ when $l\geq l_0$ for some large integer $l_{0}>0$. For such an $l$, running the geodesic $\gamma_{l}$ backwards from $y_l$ to $x$ and applying Lemma \ref{measurement}(iv) yields


\begin{equation}\label{gradient}
\textrm{grad}\;d^{+}_{y_{l}}(x_{0})=\frac{\xi_{l}}{|\xi_{l}|_{g}}.
\end{equation}
Also Lemma \ref{measurement}(v) states that the eikonal equation
$$(\textrm{grad}\;d^{+}_{y_{l}}(x),\textrm{grad}\;d^{+}_{y_{l}}(x))_{g}=-1$$
holds in $U_l$. We write this equation with the coordinates $(x^{1},\dots,x^{n},r)$:
\begin{equation}\label{eik1}
\displaystyle\sum^{n}_{i,j=1}g^{ij}\frac{\partial d^{+}_{y_{l}}}{\partial x^{i}}\frac{\partial d^{+}_{y_{l}}}{\partial x^{j}}+\left(\frac{\partial d^{+}_{y_{l}}}{\partial r}\right)^{2}=-1, \quad\quad\quad l\geq l_{0}.
\end{equation}
In particular, this relation holds on $U_{l}\cap\Sigma$. Since the function $d|_{\Sigma\times\Sigma}$ is assumed to be known, we actually know the function $d^{+}_{y_{l}}$ on $\Sigma$ (even if $y_l\in\overline{\Sigma}$ may not lie in $\Sigma$, see the remark after Lemma \ref{measurement}). This fact, together with Proposition \ref{boundary}, implies that all the tangential derivatives of $g^{ij}$ and $d^{+}_{y_{l}}$ are known on $U_{l}\cap\Sigma$. Thus we can solve for $\frac{\partial d^{+}_{y_{l}}}{\partial r}$ in \eqref{eik1} by taking a square root. The sign of the square root can be determined in the following approach. Observing that in the coordinates $(x^{1},\dots,x^{n},r)$ the last component of $\textrm{grad}\;d^{+}_{y_{l}}(x_{0})$ is $\frac{\partial d^{+}_{y_{l}}}{\partial r}(x_{0})$, and the last component of $\frac{\xi_{l}}{|\xi_{l}|_{g}}$ is $\frac{1}{l|\xi_{l}|_{g}}$ by the definition of $\xi_l$, we derive from \eqref{gradient} that
\begin{equation}\label{last_component}
\frac{\partial d^{+}_{y_{l}}}{\partial r}(x_{0})=\frac{1}{l|\xi_{l}|_{g}},
\end{equation}
which is positive. Therefore, from \eqref{eik1} we can recover $\frac{\partial d^{+}_{y_{l}}}{\partial r}$ in a neighborhood of $x_{0}$ in $U_{l}\cap\Sigma$ by taking the positive square root. Shrinking $U_{l}$ if necessary, we may assume that this neighborhood is $U_{l}\cap\Sigma$ itself.


Differentiating \eqref{eik1} with respect to a tangential direction, say $x^{m}$, we get
\begin{equation}\label{eik2}
\displaystyle\sum^{n}_{i,j=1}\left(\frac{\partial g^{ij}}{\partial x^{m}}\frac{\partial d^{+}_{y_{l}}}{\partial x^{i}}\frac{\partial d^{+}_{y_{l}}}{\partial x^{j}}+2g^{ij}\frac{\partial^{2} d^{+}_{y_{l}}}{\partial x^{i}\partial x^{m}}\frac{\partial d^{+}_{y_{l}}}{\partial x^{j}}\right)+2\frac{\partial d^{+}_{y_{l}}}{\partial r}\frac{\partial^{2} d^{+}_{y_{l}}}{\partial x^{m}\partial r}=0 \quad\quad l\geq l_{0}.
\end{equation}
From this identity we can recover $\frac{\partial^{2} d^{+}_{y_{l}}}{\partial x^{m}\partial r}$ on $U_{l}\cap\Sigma$, and similarly up to all order the tangential derivatives of $\frac{\partial d^{+}_{y_{l}}}{\partial r}$ can be recovered on $U_{l}\cap\Sigma$.

On the other hand, differentiating \eqref{eik1} with respect to $r$ we get
\begin{equation}\label{eik3}
\displaystyle\sum^{n}_{i,j=1}\left(\frac{\partial g^{ij}}{\partial r}\frac{\partial d^{+}_{y_{l}}}{\partial x^{i}}\frac{\partial d^{+}_{y_{l}}}{\partial x^{j}}+2g^{ij}\frac{\partial^{2} d^{+}_{y_{l}}}{\partial x^{i}\partial r}\frac{\partial d^{+}_{y_{l}}}{\partial x^{j}}\right)+2\frac{\partial d^{+}_{y_{l}}}{\partial r}\frac{\partial^{2} d^{+}_{y_{l}}}{\partial r^{2}}=0 \quad\quad l\geq l_{0}.
\end{equation}
Evaluating this at $x=x_{0}$ and taking \eqref{last_component} into consideration we obtain
$$
\displaystyle\sum^{n}_{i,j=1}\frac{\partial g_{ij}}{\partial r}(x_{0})\frac{\xi^{i}_{0}}{|\xi_{l}|_{g}}\frac{\xi^{j}_{0}}{|\xi_{l}|_{g}}=-\left.\left(2\sum^{n}_{i,j=1}
g^{ij}\frac{\partial^{2} d^{+}_{y_{l}}}{\partial x^{i}\partial r}\frac{\partial d^{+}_{y_{l}}}{\partial x^{j}}+2\frac{\partial^{2} d^{+}_{y_{l}}}{\partial r^{2}}\frac{1}{l|\xi_{l}|_{g}}\right)\right|_{x=x_{0}}.
$$
As $y_{0}\neq x_{0}$, $\frac{\partial^{2} d^{+}_{y_{l}}}{\partial r^{2}}(x_{0})$ remains bounded as $l\rightarrow\infty$, so the second term on the right-hand side tends to zero as $l\rightarrow\infty$, and we recover $\sum^{n}_{i,j=1}\frac{\partial g_{ij}}{\partial r}(x_{0})\xi^{i}_{0}\xi^{j}_{0}$. This completes the proof of Claim 1.\\

\emph{Claim $2$: if $y_{0}=x_{0}$, then $\sum^{n}_{i,j=1}\frac{\partial g_{ij}}{\partial r}(x_{0})\xi^{i}_{0}\xi^{j}_{0}=\sum^{n}_{i,j=1}\frac{\partial \tilde{g}_{ij}}{\partial r}(\tilde{x}_{0})\tilde{\xi}^{i}_{0}\tilde{\xi}^{j}_{0}$.}\vspace{2ex}

\begin{figure}[h]
\begin{center}
\includegraphics[width=0.8\linewidth]{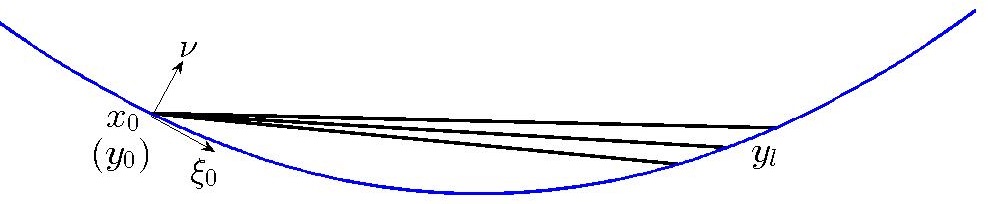}
\end{center}
\caption{The case where $y_0=x_0$.} 
\end{figure}

Notice that if $y_{0}=x_{0}$, then $y_{l}\rightarrow x_{0}$ as $l\rightarrow\infty$, so for large $l$, $\gamma_{l}$ will lie in $W\times(-\delta,\delta)$. We fix such an $l$ and pull $\tilde{g}$ back from the tubular neighborhood $\tilde{\Psi}(\widetilde{W}\times(-\delta,\delta))$ onto the tubular neighborhood $\Psi(W\times (-\delta,\delta))$ via the diffeomorphism $\eqref{diffeo}$, the pullback metric, say $g'$, has the expression
$$g'(x,r)=\sum^{n}_{i,j=1}g'_{ij}(x,r)dx^{i}dx^{j}+dr^{2}.$$
Correspondingly let $d'$ be the pullback of $\tilde{d}$ via \eqref{diffeo} onto $\Psi(W\times (-\delta,\delta))$, then we have $d(x,y)=d'(x,y)$ for all $x,y\in W$. Define $f_{ij}=g_{ij}-g'_{ij}$, in the following we will show
\begin{equation}\label{vanish}
\sum^{n}_{i,j=1}\frac{\partial f_{ij}}{\partial r}(x_{0},0)\xi^{i}_{0}\xi^{j}_{0}=0,
\end{equation}
from which Claim 2 follows. Here we have identified $x_{0}\in W$ with $(x_{0},0)\in W\times (-\delta,\delta)$.

%
%

Now we prove \eqref{vanish} by a contrapositive argument. Suppose it is not true, without loss of generality we may assume
$$\sum^{n}_{i,j=1}\frac{\partial f_{ij}}{\partial r}(x_{0},0)\xi^{i}_{0}\xi^{j}_{0}>0.$$
By continuity, there exists a conic neighborhood $V$ of $((x_{0},0);\xi_{0})$ in the tangent bundle $T\mathcal{U}$ so that
$$\sum^{n}_{i,j=1}\frac{\partial f_{ij}}{\partial r}(x,r)\xi^{i}\xi^{j}>0$$
for all $((x,r);\xi)\in V$. As $f_{ij}(x,0)=0$ by Proposition \ref{boundary}, developing $f_{ij}$ in Taylor's expansion we obtain
$$f_{ij}(x,r)=\displaystyle\frac{\partial f_{ij}}{\partial r}(x,0)r+\mathcal{O}(r^{2}),$$
thus
\begin{equation}\label{positive}
\sum^{n}_{i,j=1}f_{ij}(x,r)\xi^{i}\xi^{j}>0
\end{equation}
for all $((x,r);\xi)\in V$ with $r>0$. By choosing $l$ sufficiently large, we may assume $\{(\gamma_{l}(t);\dot{\gamma}_{l}(t)):0\leq t\leq t_{l}\}$ is contained in $V$ so that \eqref{positive} is valid. Since $\xi_{0}\in T_{x_{0}}\Sigma$ is timelike with respect to $g$, $\xi_{l}\in T_{x_{0}}\Sigma$ is close to $\xi_{0}$, and $g=g'$ on $T_{x_{0}}\Sigma$, we see that $\gamma_{l}$ is a timelike curve with respect to $g'$ for large $l$. (Notice that by definition $\gamma_{l}$ is a timelike geodesic with respect to $g$. The argument here shows that $\gamma_{l}$ is also a timelike curve for $g'$, but not necessarily a timelike geodesic.) We assume the fixed $l$ is chosen to be so large that $\gamma_{l}$ is indeed a timelike curve with respect to $g'$. Therefore, for this timelike curve, we can find a strictly increasing smooth parametrization $\iota:[0,\ell]\rightarrow [0,t_{l}]$ such that the reparametrized curve
$$\Gamma:=\gamma_{l}\circ\iota:[0,\ell]\rightarrow\mathcal{U}$$
satisfies $(\dot{\Gamma}(t),\dot{\Gamma}(t))_{g'}=-1$ for all $t$. It follows from \eqref{positive} that
\begin{equation}\label{contradiction1}
\begin{array}{rl}\vspace{1ex}
& \displaystyle\int^{\ell}_{0}(\dot{\Gamma}(t),\dot{\Gamma}(t))_{g}\,dt+\ell \\ \vspace{1ex}
=& \displaystyle\int^{\ell}_{0}\sum^{n}_{i,j=1}g_{ij}(\Gamma(t))\dot{\Gamma}^{i}(t)\dot{\Gamma}^{j}(t)\,dt-
\displaystyle\int^{\ell}_{0}\sum^{n}_{i,j=1}g'_{ij}(\Gamma(t))\dot{\Gamma}^{i}(t)\dot{\Gamma}^{j}(t)\,dt\\ \vspace{1ex}
=& \displaystyle\int^{\ell}_{0}\sum^{n}_{i,j=1}f_{ij}(\Gamma(t))\dot{\Gamma}^{i}(t)\dot{\Gamma}^{j}(t)\,dt > 0.
\end{array}
\end{equation}
On the other hand, let $\gamma'_{l}$ be the pullback via \eqref{diffeo} of the unique radial geodesic in $\widetilde{\mathcal{U}}$ joining $\tilde{y}_{l}$ and $\tilde{x}_{0}$, hence $\gamma'_{l}$ is, with respect to $g'$, the longest timelike curve joining $y_{l}$ and $x_{0}$ in $\Psi(W\times (-\delta,\delta))$. Therefore, we conclude
$$\ell\leq d'(y_{l},x_{0})=d(y_{l},x_{0})=L(\Gamma).$$
By Cauchy-Schwarz inequality
\begin{equation}\label{contradiction2}
\ell^{2}\leq L^{2}(\Gamma)=\left(\displaystyle\int^{\ell}_{0}|\dot{\Gamma}(t)|^{\frac{1}{2}}_{g}\right)^{2}\,dt
\leq  -\ell\displaystyle\int^{\ell}_{0}(\dot{\Gamma}(t),\dot{\Gamma}(t))_{g}\,dt
\end{equation}
since $(\dot{\Gamma}(t),\dot{\Gamma}(t))_{g}<0$ for all $t$. From \eqref{contradiction2} we derive
$$\ell+\displaystyle\int^{\ell}_{0}(\dot{\Gamma}(t),\dot{\Gamma}(t))_{g}\,dt\leq 0,$$
which contradicts \eqref{contradiction1}. This completes the proof of identity \eqref{vanish}, hence Claim $2$.\\

Combining \emph{Claim 1} and \emph{Claim 2}, in either case $\sum^{n}_{i,j=1}\frac{\partial g_{ij}}{\partial r}(x_{0})\xi^{i}_{0}\xi^{j}_{0}$ is uniquely determined. To recover $\frac{\partial g_{ij}}{\partial r}(x_{0})$, we need to perturb $\xi_{0}$: for any $\xi\in T_{x_{0}}\Sigma$ which is sufficiently close to $\xi_{0}$, we run the above arguments to recover $\sum^{n}_{i,j=1}\frac{\partial g_{ij}}{\partial r}(x_{0})\xi^{i}\xi^{j}$, these values are enough to determine the matrix $(\frac{\partial g_{ij}}{\partial r}(x_{0}))_{1\leq i,j\leq n}$, hence we obtain $\frac{\partial g_{ij}}{\partial r}(x_{0})$. In particular, evaluating at $x_{0}=\hat{x_{0}}$ completes the proof of \eqref{normal} for $k=1$.\\

\emph{Step 2:} In this step, we prove \eqref{normal} for $k\geq 2$. This is an inductive argument. However, to make the idea clear, we state the proof only for $k=2$. It should be obvious how this inductive process is done for any $k\geq 3$.

If $y_{0}\neq x_{0}$, differentiate \eqref{eik3} with respect to $r$ and evaluate at $x=x_{0}$. Since we have known $\frac{\partial g_{ij}}{\partial r}$ (note that in Step 1 we actually found $\frac{\partial g_{ij}}{\partial r}(x_{0})$ for all $x_{0}$ near $\hat{x_{0}}$, not just $\frac{\partial g_{ij}}{\partial r}(\hat{x_{0}})$), from \eqref{eik3} we can recover $\frac{\partial^{2} d^{+}_{y_{l}}}{\partial r^{2}}$ as well as all its tangential derivatives on $U_{l}\cap\Sigma$. The only unknown term will be $\frac{\partial^{3} d^{+}_{y_{l}}}{\partial r^{3}}(x_{0})$, but it will be multiplied by $\frac{\partial d^{+}_{y_{l}}}{\partial r}(x_{0})=\frac{1}{l|\xi_{l}|_{g}}$. Taking the limit $l\rightarrow\infty$ will recover $\sum^{n}_{i,j=1}\frac{\partial^{2} g_{ij}}{\partial r^{2}}(x_{0})\xi^{i}_{0}\xi^{j}_{0}$ as above.

If $y_{0}=x_{0}$, as in the proof of \emph{Claim 2} we assume $\sum^{n}_{i,j=1}\frac{\partial^{2} f_{ij}}{\partial r^{2}}(x_{0})\xi^{i}_{0}\xi^{j}_{0}\neq 0$, without loss of generality we may assume it is positive. Writing the Taylor expansion of $f_{ij}$ with respect to $r$ at $r=0$, by Proposition \ref{boundary} and the case $k=1$ in \eqref{normal}, we see that \eqref{positive} holds for $((x,r);\xi)$ in a conic neighborhood of $((x_{0},0);\xi_{0})$ with $r>0$. Similarly we get a contradiction as above.

In either case, we can uniquely determine $\sum^{n}_{i,j=1}\frac{\partial^{2} g_{ij}}{\partial r^{2}}(x_{0})\xi^{i}_{0}\xi^{j}_{0}$. Finally, by first varying $\xi_{0}$, and next varying $x_{0}$, and then evaluating at $x_{0}=\hat{x_{0}}$, we prove \eqref{normal} when $k=2$. Applying this construction inductively completes the proof of the theorem.
\end{proof}

\section{Global determination of the manifold}\label{sec_glob}

In this section we describe a procedure to obtain
a maximal real-analytic extensions of a real-analytic manifolds
that are
geodesically complete modulo scalar curvature singularities.

Let $(M,g)$ and $(\tilde M, \tilde g)$ be a smooth Lorentzian manifolds,
and let $\phi : U \to \tilde U$
be an isometry of an open set $U \subset M$ onto an open set $\tilde U \subset \tilde M$.
Let $\gamma : [0,\ell] \to M$ be a path starting from $U$, that is, $\gamma(0) \in U$.
Let $I$ be a connected neighborhood of zero in $[0,\ell]$.
We say that a family $\phi_t : U_t \to \tilde U_t$, $t \in I$,
is a continuation of $\phi$ along $\gamma$ if
\begin{itemize}
\item[(i)] $U_t \subset M$ and $\tilde U_t \subset \tilde M$ are open, $\gamma(t) \in U_t$ and $\phi_t$ is an isometry,
\item[(ii)] for all $s \in I$ there is $\epsilon > 0$ such that
$\phi_t = \phi_s$ in $U_t \cap U_s$ whenever $|t-s| < \epsilon$, and 
\item[(iii)] $\phi_0 = \phi$ in $U_0 \cap U$.
\end{itemize}
We say that $\phi$ is extendable along $\gamma$ if there is a continuation of $\phi_t$, $t \in [0,\ell]$, along $\gamma$.


We recall that a continuous path $\gamma : [0,\ell] \to M$ is a broken geodesic if there are $0 < l_0 < l_1 <  \dots <  l_N < \ell$
such that $\gamma$ is a geodesic on $[l_{j-1},l_j]$, $j=1,2,\dots,N$.

\begin{theorem}
\label{thm_glueing}
Suppose that $(M,g)$ and $(\tilde M, \tilde g)$ are connected. 
Let $\phi : U \to \tilde U$ be an isometry of an open set $U \subset M$ onto an open set $\tilde U \subset \tilde M$,
and suppose that $\phi$ is extendable along all broken geodesics $\gamma : [0,\ell] \to M$ starting from $U$.
Suppose, furthermore, that all broken geodesics $\gamma : [0,\ell) \to M$, $\ell \in (0,\infty)$,
starting from $U$ satisfy the following:
\begin{itemize}
\item[(L)] If $\phi_t$ is a continuation of $\phi$ along $\gamma$
and the limit $\lim_{t \to \ell} \phi_t(\gamma(t))$ exists, then the limit $\lim_{t \to \ell} \gamma(t)$ exists.
\end{itemize}
Then $(M,g)$ and $(\tilde M, \tilde g)$ have the same universal Lorentzian covering space.
\end{theorem}

Although the assumptions in the theorem may seem unsymmetric with respect to $M$ and $\tilde M$, in fact, they are not. The extendability up to the end point $\ell$ implies the condition (L) with the roles of $M$ and $\tilde M$ interchanged. We will also see below that if $(M,g)$ and $(\tilde M, \tilde g)$
are geodesically complete modulo scalar curvature singularities and real-analytic,
and if there is a local isometry $\phi : U \to \tilde U$
as above,
then $(M,g)$ and $(\tilde M, \tilde g)$ satisfy the assumptions of the theorem.

\begin{proof}
We begin by constructing a matched covering for $M$, see \cite[p. 203]{O} for the definition.
Let $p \in U$. 
We denote by $A$ the set of pairs $(\gamma, t)$ where $\gamma : [0, \ell] \to M$ is a broken geodesic satisfying $\gamma(0) = p$
and $t \in [0, \ell]$.
Let $(\gamma, t) \in A$.
Let us choose a continuation $\phi_t^\gamma : V_t^\gamma \to \tilde M$ of $\phi$ along $\gamma$.
The sets $\{V_t^\gamma: (\gamma, t) \in A\}$, form an open covering of $M$
since any two points of $M$ can be joined by a broken geodesic, see e.g. \cite[Lem. 3.32]{O}.
We may choose a smooth Riemannian metric tensor $g^+$ on $M$.
We choose a neighborhood $U_t^\gamma \subset V_t^\gamma$ of $\gamma(t)$ such that $U_t^\gamma$ is convex with respect to $g^+$.
This is possible since there is a lower bound for the strong convexity radius on any compact set in $M$, see e.g. \cite[Th. IX.6.1]{Chavel2006}.
Then the intersections $U_t^\gamma \cap U_s^\mu$ are connected for all $(\gamma, t),\, (\mu, s) \in A$ whenever non-empty.



We define a relation $(\gamma, t) \sim (\mu, s)$ on $A$ by
$$
\text{$U_t^\gamma \cap U_s^\mu \ne \emptyset$ and $\phi_t^\gamma = \phi_s^\mu$ in $U_t^\gamma \cap U_s^\mu$}.
$$
Let $(\gamma, t) \sim (\mu, s)$, $(\mu, s) \sim (\beta, r)$ and $W := U_t^\gamma \cap U_s^\mu \cap U_r^\beta \ne \emptyset$.
Then there is $q \in W$ and $d \phi_t^\gamma = d\phi_r^\beta$ at $q$.
Thus $\phi_t^\gamma = \phi_r^\beta$ in the connected set $U_t^\gamma \cap U_r^\beta$ by \cite[Lem. 3.62]{O}.
That is $(\gamma, t) \sim (\beta, r)$. Hence
the open sets $U_t^\gamma$, $(\gamma, t) \in A$,
together with the relation $\sim$ give a matched covering of $M$.

To simplify the notation, we write for $a = (\gamma, t)$
$$
\mathcal U_a := U_t^\gamma, \quad \Phi_a := \phi_t^\gamma.
$$
Following \cite[p. 203]{O},
we define $\mathcal X = \{(y, a) \in M \times A:\ y \in \mathcal U_a \}$
and an equivalence relation $(y, a) \approx (z, b)$ on $X$ by
$$
\text{$y = z$ and $a \sim b$}.
$$
We let $X = \mathcal X / \approx$ and denote the equivalence classes by $[(y,a)]$.
We equip $A$ with the discrete topology, $\mathcal X$ with the product topology and $X$ with the quotient topology.
Moreover, we equip $X$ with the unique maximal manifold structure such that each
$$
\lambda_a : \mathcal U_a \to X, \quad \lambda_a(y) := [(y,a)], \quad a \in A,
$$
is a diffeomorphism onto a domain in $X$.
We set $F([(y,a)]) = y$ and get a local diffeomorphism $F : X \to M$ such that
$$
F = \lambda_a^{-1}, \quad \text{on $\lambda_a(\mathcal U_a)$}.
$$
We equip $X$ with the pullback metric $F^* g$. Then $F : X \to M$ is a local isometry.

The map $\tilde F([(y,a)]) := \Phi_a(y)$ is well defined. Indeed, if $(y,a) \approx (z,b)$ then $y = z \in \mathcal U_a \cap \mathcal U_b$ and
$\Phi_a(y) = \Phi_b(z)$. 
Moreover,
$$
\tilde F([(y,a)]) = \Phi_a(y) = \Phi_a(F([(y,a)])),
$$
whence $\tilde F = \Phi_a \circ F$ on $\lambda_a(\mathcal U_a)$, and $\tilde F : X \to \tilde M$ is a local isometry.


To finish the proof of the theorem we will need to invoke the following two lemmas several times.
In the formulation of the lemmas we assume implicitly the facts that we have established so far in the proof.

\begin{lemma}
\label{lem_equivalence_in_U_a}
Let $a = (\gamma,t) \in A$, $[(y,a)] \in X$, and let us consider a broken geodesic
$\rho : [0,\ell] \to \mathcal U_a$ satisfying $\rho(0) = \gamma(t)$.
We denote by $\mu$ the concatenation of $\gamma$ and $\rho$.
Then $(\gamma,t) \sim (\mu,s)$ for all $s \in [t,t+\ell]$.
\end{lemma}
\begin{proof}
Notice that the set
$$
I = \{s \in [t,t+\ell];\ (\gamma,t) \sim (\mu,s)\}
$$
is nonempty since $t \in I$. It is enough to show that $I$ is closed and open.
By \cite[Lem. 3.62]{O} we have
$$
I = \{s \in [t,t+\ell];\ \text{$\phi_t^\gamma(\mu(s)) = \phi_s^\mu(\mu(s))$ and $d\phi_t^\gamma = d\phi_s^\mu$ at $\mu(s)$}\}
$$
We have $\phi_s^\mu(x) = \phi_r^\mu(x)$ for $s$ near $r$ and $x$ near $\mu(s)$.
Thus the maps $s \mapsto \phi_s^\mu(\mu(s))$ and $s \mapsto d\phi_s^\mu|_{\mu(s)}$ are smooth and $I$ is closed.

In order to show that $I$ is open, let us suppose that $(\gamma,t) \sim (\mu,s)$.
By the definition of a continuation $(\mu,s) \sim (\mu,r)$ for $r$ close to $s$.
The definition of a matched covering implies that $(\gamma,t) \sim (\mu,r)$ since $\mu(r)$ is in $U_t^\gamma \cap U_s^\mu \cap U_r^\mu$.
Thus $I$ is open.
\end{proof}

\begin{lemma}
\label{lem_lift}
Let $\mu : [0,\ell] \to M$ be a broken geodesic satisfying $\mu(0) = p$.
Then there are
$$0 = t_0 < t_1 < \dots < t_N = \ell$$
such that $\mu([t_{j-1},t_j]) \subset \mathcal U_{a_j}$, $j=1,\dots,N$,
where $a_j = (\mu, t_{j-1})$,
and we may define a continuous path $\hat \mu : [0,\ell] \to X$ by
$$
\hat \mu(\tau) = \lambda_{a_j}(\mu(\tau)), \quad t \in [t_{j-1},t_j].
$$
Moreover, $F \circ \hat \mu = \mu$, that is, $\hat \mu$ is a lift of $\mu$ through $F$.
\end{lemma}
\begin{proof}
Compactness of $\mu([0,\ell])$ implies the existence of the intervals $[t_{j-1},t_j]$,
and Lemma \ref{lem_equivalence_in_U_a} implies $(\mu,t_{j-1}) \sim (\mu,t_j)$.
Hence $\lambda_{a_j}(\mu(t_j)) = \lambda_{a_{j+1}}(\mu(t_j))$ and $\hat \mu$ is continuous.
\end{proof}

Let us now return to the proof of the theorem.
We will show next that $F$ is a covering map.
By \cite[Th. 7.28]{O} it is enough to show that all geodesics of $M$ can be lifted through $F$.
Let $\sigma : [0,\ell] \to M$ be a geodesic, let $[(y,a)] \in X$, $a = (\gamma,t)$,
and suppose that $\sigma(0) = y$.
There is a broken geodesic $\rho : [0,r] \to \mathcal U_a$ from $\gamma(t)$ to $y$.
We denote by $\mu$ the concatenation of $\gamma$, $\rho$ and $\sigma$,
and by $\hat \mu$ the lift of $\mu$ as in Lemma \ref{lem_lift}.
Now $\hat \sigma(\tau) = \hat \mu(\tau + t + r)$ is a lift of $\sigma$.
Let $j$ be the index satisfying $t+r \in [t_{j-1},t_j)$. Then
$$
\hat \sigma(0) = \hat \mu(t + r) = [(y, (\mu, t_{j-1}))] = [(y, (\mu, t+r))] = [(y, (\gamma, t))],
$$
where we have employed Lemma \ref{lem_equivalence_in_U_a} twice.
We have shown that $F$ is a covering map.

Let us show next that $\tilde F$ is a covering map.
Let $\sigma : [0,\ell] \to \tilde M$ be a geodesic, let $[(y,a)] \in X$, $a = (\gamma,t)$,
and suppose that $\sigma(0) = \Phi_a(y)$.
There is a broken geodesic $\rho : [0,r] \to \mathcal U_a$ from $\gamma(t)$ to $y$.
We denote by $\alpha$ the concatenation of $\gamma$ and $\rho$ and write $s = t + r$.
Let $\beta : [0,L) \to M$ be the maximal geodesic satisfying $\beta(0) = y$
and $d \phi_s^\alpha \dot \beta(0) = \dot \sigma(0)$.
Moreover, let $\mu$ be the concatenation of $\alpha$ and $\beta$.
Then the geodesic $\tilde \sigma(\tau) = \phi_{s+\tau}^\mu(\mu(s+\tau))$
coincides with $\sigma$ on $[0,\ell] \cap [0,L)$ since both the geodesics have the same initial data.
If $\ell \ge L$ then the limit $\lim_{t \to L} \beta(t)$ exists by (L)
but this is a contradiction with the maximality of $\beta$.	
Thus $\ell < L$ and $\tilde \sigma = \sigma$ on $[0,\ell]$.

Let $\hat \mu$ be the lift of $\mu$ as in Lemma \ref{lem_lift}.
Then for $t \in [t_{j-1},t_j]$
$$
\tilde F \circ \hat \mu(t) = \Phi_{a_j} \circ F \circ \hat \mu(t) = \phi^\mu_{t_{j-1}} \circ \mu(t) = \phi_t^\mu \circ \mu(t).
$$
Indeed, the first identity follows from the definition of $\tilde F$,
the second from the fact that $\hat \mu$ is a lift of $\mu$, and the last from Lemma \ref{lem_equivalence_in_U_a}.
Hence $\hat \sigma(\tau) = \hat \mu(\tau + s)$ is a lift of $\sigma$ through $\tilde F$.
As above we see that
$$
\hat \sigma(0) = [(y, a)].
$$
We have shown that $\tilde F$ is a covering map.

Let us show that $X$ is connected.
It is enough to show that a point $[(y,a)] \in X$ can be connected to $[(p,b)] \in X$ where $a = (\gamma,t)$ and $b = (\beta, 0)$.
There is a broken geodesic $\rho : [0,\ell] \to \mathcal U_a$ from $\gamma(t)$ to $y$.
We denote by $\mu$ the concatenation of $\gamma$ and $\rho$, and by $\hat \mu$ the lift of $\mu$ as in Lemma \ref{lem_lift}.
Now
$$
\hat \mu(t+\ell) = [(y, (\mu, t_{N-1}))] = [(y, (\mu, t+\ell))] = [(y, (\gamma, t))].
$$
We have shown that $X$ is connected.

As $X$ is connected, it has the universal covering $\tilde X$. Moreover, as the composition of two covering maps is a covering map in the case of manifolds, 
the covering $\tilde X$ is the universal covering of $M$ and $\tilde M$.

%
\end{proof}

\begin{lemma}
\label{lem_equivalence_limits}
Suppose that $(M,g)$ and $(\tilde M, \tilde g)$ are geodesically complete modulo scalar curvature singularities.
Let $\phi : U \to \tilde U$
be an isometry of an open set $U \subset M$ onto an open set $\tilde U \subset \tilde M$.
Let $\ell \in (0,\infty)$ and let $\gamma : [0,\ell) \to M$ be a broken geodesic such that $\gamma(0) \in U$.
Suppose that $\phi_t$, $t \in [0,\ell)$, is a continuation of $\phi$ along $\gamma$,
and define $\mu(t) = \phi_t(\gamma(t))$, $t \in [0,\ell)$. Then the limit $\lim_{t \to \ell} \mu(t)$ exists
if and only if the limit $\lim_{t \to \ell} \gamma(t)$ exists.
\end{lemma}
\begin{proof}
Local isometries map geodesics to geodesics and $\mu(t) = \phi_s(\gamma(t))$ for $t$ near $s \in [0,r)$.
Thus $\mu$ is a broken geodesic.

Notice that the limit $\lim_{t \to \ell} \gamma(t)$ exists if and only if
$\kappa(\gamma(t))$ is bounded as $t \to \ell$
for all scalar curvature invariants $\kappa$ of $(M,g)$.
 {Indeed,
If $\gamma(t) \to x$ in $M$ as $t \to \ell$, then $\kappa(\gamma(t))$ is bounded as $t \to \ell$
for all scalar curvature invariants $\kappa$ of $(M,g)$ since $\kappa$ is smooth near $x$.
On the other hand if the limit $\lim_{t \to \ell} \gamma(t)$ does not exist, then $\gamma$ can not be extended, and
there is a scalar curvature invariant $\kappa$ of $(M,g)$ such that $\kappa(\gamma(t))$ is unbounded as $t \to \ell$.
}
The analogous statement holds for the limit $\lim_{t \to \ell} \mu(t)$.
The claim follows, since if $\tilde\kappa$ is a scalar curvature invariant of $(\tilde M, \tilde g)$
then the corresponding scalar curvature invariant $\kappa$ of $(M, g)$ satisfies $\kappa(x) = \tilde \kappa(\phi_t(x))$
for all $x \in U_t$.
\end{proof}

\begin{lemma}
\label{ac_gamma}
Suppose that $(M,g)$, $(\tilde M, \tilde g)$ and $\phi : U \to \tilde U$
are as in Lemma \ref{lem_equivalence_limits},
and let $\gamma : [0,\ell] \to M$ be a broken geodesic starting from $U$.
Suppose, furthermore, that $(M,g)$ and $(\tilde M, \tilde g)$ are real-analytic.
Then $\phi$ is extendable along $\gamma$.
\end{lemma}
\begin{proof}
\def\U{\mathcal U}
Let $S \ge 0$ be the supremum of $s \in [0,\ell]$ such that there is a continuation $\phi_t : U_t \to \tilde U_t$, $t \in [0,s]$, of $\phi$.
We define $\mu(t) := \phi_t(\gamma(t))$, $t \in [0,S)$.
Lemma \ref{lem_equivalence_limits} implies that the limit $\lim_{t \to S} \mu(t)$ exists. We denote the limit by $\mu(S)$.
Let $\U$ and $\tilde \U$ be simply convex neighborhoods of $\gamma(S)$ and $\mu(S)$ respectively.
Let $\epsilon > 0$ be such that $\gamma(t) \in \U$ and $\mu(t) \in \tilde \U$ for all $t \in [S-\epsilon, S]$.
By decreasing $\epsilon$ we may also assume that both $\gamma$ and $\mu$ are geodesics on $[S-\epsilon, S]$.

We write $s = S - \epsilon$, $p = \gamma(s)$ and $q = \mu(s)$.
We will work in the normal coordinates around $p$ and $q$.
In the normal coordinates, the isometry $\phi_s$ coincides with the linear map $d \phi_s|_p$ in $U_s \cap \U$,
see e.g. \cite[p. 91]{O}.
In the normal coordinates, the simply convex neighborhoods $\U$ and $\tilde \U$
are neighborhoods of the origins in $T_p M$ and $T_q \tilde M$ respectively.
We define $W$ to be the connected component of $\U \cap d \phi_s|_p^{-1} (\tilde \U)$ that contains the origin,
and denote by $\psi$ the linear map $d \phi_s|_p$ on $W$.
Let $X$ and $Y$ be real-analytic vector fields on $W$. Then
$$
(d \psi X, d \psi Y)_{\tilde g \circ \psi} = (X, Y)_g
$$
in $U_s \cap W$ since $\psi$ is an isometry there.
The both sides of the above identity are real-analytic functions on the connected set $W$,
whence the identity holds on $W$, see e.g. \cite[Lem. VI.4.3]{Helgason1962}.
Thus $\psi$ is an isometry of $W$ onto an open set in $\tilde M$.

We write $v = \dot \gamma(s)$.
In the normal coordinates, the geodesic $\gamma|_{[s, S]}$ has the form $\gamma(t-s) = (t-s)v$,
and the geodesic $\mu|_{[s, S]}$ has the form $\mu(t-s) = (t-s) d\phi_s|_p v$.
Thus $[s,S]v \subset U$ and $d\phi_s ([s,S]v) \subset \tilde U$. In particular, $\gamma(S) = (S-s)v \in W$.
If $S < \ell$, then there is $\delta  > 0$ such that $\gamma(t) \in W$ for $t \in [S,S+\delta]$ since $W$ is open.
Moreover, there is $r  \in (s,S)$ such that $\gamma(t) \in U_s$ for $t \in [s,r]$.
We define $V$ to be the connected component of $W \cap U_s$ that contains $\gamma(s)$.
Now
$$
\psi_t =
\begin{cases}
\phi_t : U_t \to \tilde U_t, & t \in [0,s)
\\\psi : V \to \psi(V), & t \in [s,r),
\\\psi : W \to \psi(W), & t \in [r,S+\delta),
\end{cases}
$$
is a continuation of $\phi$.
{\footnotesize
Indeed, if $t < s$ is close to $s$, then $\phi_t = \phi_s$ in $U_t \cap U_s$.
Hence $\phi_t = \psi$ in $U_t \cap V$.
Moreover, $\gamma(r) \in V$ and $\psi_t = \psi_{r}$ in $V \cap W = V$
if $t < r$ is close to $r$.
}
But this is a contradiction with maximality of $S$. Hence $S = \ell$.
\end{proof}

{Now we are ready to prove our second main theorem.}

\begin{proof}[Proof of Theorem \ref{thm_global}]
The metric tensors $g$ and $\tilde g$
are real-analytic in geodesic normal coordinates, see e.g. \cite[Th. 2.1]{DeTurck1981}.
Theorem \ref{thm_local} guarantees that there is a linear bijection $L: T_p M \to T_{\tilde p} \tilde M$
such that if $V_0, \dots, V_n$ is a basis of $T_p M$ and we define $\tilde V_j = L V_j$, then
the Taylor coefficients of the metric tensors $g$ and $\tilde g$ coincide in the normal coordinates
\begin{align*}
\psi(x^0, \dots, x^n) = \exp_p(x^j V_j), \quad \tilde \psi(x^0, \dots, x^n) = \widetilde\exp_{\tilde p}(x^j \tilde V_j)
\end{align*}
defined on $\mathcal B = \{x \in \R^{1+n};\ (x^0)^2 + \dots + (x^n)^2 < r\}$ where $r > 0$ is small enough.
As $g$ and $\tilde g$ are real-analytic, they coincide in these coordinates.
Hence $\phi = \tilde \psi \circ \psi^{-1}$ is an isometry of $U = \psi(\mathcal B) \subset M$ onto $\tilde U = \tilde \psi(\mathcal B) \subset M$.
The claim follows from Theorem \ref{thm_glueing}
together with Lemmas \ref{lem_equivalence_limits} and \ref{ac_gamma}.
\end{proof}
\medskip

Next we prove Proposition \ref{prop: Analytic stationary}. We show that stationary solutions of the Einstein field equations coupled with scalar fields are real-analytic.

\begin{proof}[Proof of Proposition \ref{prop: Analytic stationary}]
Given $p\in M$, first we show that there are local coordinates $y = (y^0, \dots, y^3)$ near $p$ such that $Z = \p_{y^0}$ and that $\hat g = (g^{jk})_{j,k=1}^3$ is positive definite.

We start with the coordinates
$$
(\tilde y^0, \dots, \tilde y^3) \mapsto \exp_p(\tilde y^0 Z(p) + \tilde y^1 V_1 + \tilde y^2 V_2 + \tilde y^3 V_3), 
$$
where $Z(p)/|Z(p)|_g, V_1, V_2, V_3$ form an orthonormal basis of $T_p M$. Here $Z(p)\neq 0$ since $Z$ is timelike. In these coordinates $Z$ can be written as 
$$Z=Z^0 \partial_{\tilde{y}^0}+Z^1 \partial_{\tilde{y}^1}+Z^2 \partial_{\tilde{y}^2}+Z^3 \partial_{\tilde{y}^3}$$
with $Z^0(p)=1$ and $Z^{1}(p)=Z^2 (p)=Z^3 (p)=0$; and in these coordinates the metric $\tilde{g}$ at $p$ is diagonal with diagonal elements $(-|Z(p)|_{g},1,1,1)$. In the following, we write $\tilde{y}'=(\tilde{y}^1,\tilde{y}^2,\tilde{y}^3)$ and use analogous notations also for other quantities. 
Denote the flow of $Z$ by $\varphi_{t}$, and define a smooth map
$$(t,y') \mapsto \varphi_{t}(y')=(\tilde{y}^0,\tilde{y}').$$ 
This map is indeed a diffeomorphism near $p$. To see this, simply notice that
$$
D\tilde y/D(t,y') = \begin{pmatrix}
Z^0 & D\tilde y^0/Dy' \\
Z'	& D\tilde y'/Dy'
\end{pmatrix},
$$
which is the identity at the origin. Thus we can choose $(t,y')$ as local coordinates near $p$. Renaming $t$ as $y^0$ gives $\partial_{y^0}=Z$. Moreover, in the coordinates $y=(y^0,y')$ the metric $g$ can be written as
$$g=\left(\frac{D\tilde{y}}{Dy}\right)^{T}\tilde{g}\left(\frac{D\tilde{y}}{Dy}\right)$$
which, from our analysis above, is diagonal with diagonal elements $(-|Z(p)|_{g},1,1,1)$ at $p$. It follows that the matrix $\hat{g}=(g^{jk})^{3}_{j,k=1}$ is positive definite at $p$, and hence by continuity is also positive definite in a neighborhood of $p$.

%

Now we choose the coordinates $y = (y^0, \dots, y^3)$ as above. As $Z = \p_{y^0}$ is a Killing field, we have $\p_{y^0} g^{pq} = 0$ in the coordinates $y$, and the wave operator has the form
$$
\Box_g u = \sum_{p,q=1}^3 |g|^{-1/2} \p_{y^p} (|g|^{1/2} \hat g^{pq} \p_{y^q} u) 
$$
if $u$ is a function of $y'$ only.
Note also that
$$
\Box_g y^0 = \sum_{p=1}^3 |g|^{-1/2} \p_{y^p} (|g|^{1/2} \hat g^{p0})
$$  
is a function of $y'$ only.
Let us choose functions $x^j(y')$, $j=1,2,3$, solving the elliptic problem
$$
\Box_g x^j = 0, \quad \p_{y^k} x^j(0) = \delta_k^j, \quad k = 1,2,3.
$$
Moreover, let us choose a function $h(y')$ solving the problem
$$
\Box_g h = - \Box_g y^0,
\quad \p_{y^k} h(0) = 0, \quad k = 1,2,3.
$$
and define $x^0 = y^0 + h(y')$. 
Then $Dx/Dy$ is the identity at the origin,
and $x = (x^0, x')$ give local coordinates. 
Moreover, $Z = \p_{x^0}$ and the coordinates $x$ are harmonic, that is, $\Box_g x^j = 0$, $j=0,1,2,3$.

Note that Einstein equations are equivalent to 
\begin{align*}
\Ric(g)=\rho,\quad \rho_{jk}=T_{jk}-\frac 12 ((g)^{  nm}T_{nm})g_{jk}+2\Lambda g_{jk},\quad \hbox{on }M,
 \end{align*}
see \cite[p.\ 44]{ChBook} and we recall (see \cite{FM,HKM}) that
\begin{align} \label{q-formula2copy AAA}
\Ric_{\mu\nu}(g)&= \Ric_{\mu\nu}^{(h)}(g)
+\frac 12 (g_{\mu q}\frac{\p \Gamma^q}{\p x^{\nu}}+g_{\nu q}\frac{\p \Gamma^q}{\p x^{\mu}})
\end{align}
where  $\Gamma^q=g^{mn}\Gamma^q_{mn}$,
\begin{align} \label{q-formula2copyBa}
& \hspace{-1cm}\Ric_{\mu\nu}^{(h)}(g)=
-\frac 12 g^{pq}\frac{\p^2 g_{\mu\nu}}{\p x^p\p x^q}+ P_{\mu\nu},
\\ \nonumber
& \hspace{-2cm}P_{\mu\nu}=
g^{ab}g_{ps}\Gamma^p_{\mu b} \Gamma^s_{\nu a}+
\frac 12(\frac{\p g_{\mu\nu }}{\p x^a}\Gamma^a
+ \nonumber
g_{\nu l}  \Gamma^l _{ab}g^{a q}g^{bd}  \frac{\p g_{qd}}{\p x^\mu}+
g_{\mu l} \Gamma^l _{ab}g^{a q}g^{bd}  \frac{\p g_{qd}}{\p x^\nu}).\hspace{-2cm}
\end{align}
Note that $P_{\mu\nu}$ is a polynomial of $g_{jk}$, $g^{jk}$ and the first derivatives of $g_{jk}$.
In the harmonic coordinates $x$, we have $\Gamma^q=0$, $q=0,1,2,3$,
and thus $\Ric_{\mu\nu}(g)$ coincides with $\Ric_{\mu\nu}^{(h)}(g)$.

As $Z = \p_{x^0}$, the stationarity of $g$ and $\phi$ implies that 
the equation (\ref{eq: Ein1}),  (\ref{eq: Ein2}), and  (\ref{eq: Ein3}),
have the form
\begin{align*}
  &\hspace{-1cm}  -\sum_{p,q=1}^3\frac 12 \hat g^{pq} \frac{\p^2 g_{jk}}{\p x^p \p x^q}+
    P_{jk}(g,\p g) =\rho_{jk}(g, T), \\
    &\hspace{-.5cm}T_{jk}=\bigg(\sum_{\ell=1}^L\p_j\phi_\ell \,\p_k\phi_\ell
-\frac 12 g_{jk}g^{pq}\p_p\phi_\ell\,\p_q\phi_\ell\bigg)-{\mathcal V}( \phi)
g_{jk},
\\
  &\hspace{-1cm}
    -\sum_{p,q=1}^3\frac 12 \hat g^{pq}(x)\frac{\p^2 \phi_\ell}{\p x^p \p x^q}
 -\mathcal V_\ell(\phi)=0,
\end{align*}

in the coordinates $x$.
This is an elliptic non-linear system of equations,
where $P_{jk}$, $\rho_{jk}$ and $\mathcal V$ are real-analytic.
By Morrey's theorem
\cite{Morrey}, a smooth solution of a real-analytic elliptic
system is real-analytic. Thus
$g$ and $\phi$ are real analytic in the coordinates $x$,
and also in the geodesic normal coordinates.

Let us now consider the differentiable structure of $M$ given by
the atlas of  convex normal coordinates associated to $g$. The transition
functions between such coordinates are real-analytic, and thus
$(M,g)$ can be considered as a real-analytic manifold.
\end{proof}

Finally, we give the proof of Corollary \ref{cor_global 2}.

\begin{proof}[Proof of Corollary \ref{cor_global 2}]
{By Proposition \ref{prop: Analytic stationary},
the manifolds $(M,g)$ and $(\tilde M,\tilde g)$ are   real-analytic
and $\phi$ and $\tilde \phi$ are   real-analytic functions on these manifold.
By Theorem   \ref{thm_global}, the simply connected manifolds $(M,g)$
and $(\tilde M,\tilde g)$ are isometric and there is a real-analytic isometry $F:M\to\tilde M$.

{\mltext Next we consider the Killing fields.
Let 
$U\subset T\Sigma$ be a neighborhood of $(\hat x_0,\hat \xi_0)$ and
$\epsilon>0$ for which the condition {\bf H} is valid. We may assume that 
$\hat \xi_0$ is a past-pointing timelike vector.

Pick past-pointing vectors $\eta_{j}\in T_{\hat{x}_0}M$, $j=1,\dots,n+1$ so that they are linearly independent and so that $\gamma_{\hat x_0,\eta_j}(\ell_j)\in \overline \Sigma$ for some $\ell_j>0$. The latter condition can be achieved by the hypothesis \textbf{H} as long as $\eta_j$ is sufficiently close to its projection on $T_{\hat{x}_0}\Sigma$. 
Denote $\gamma_j:=\gamma_{\hat x_0,\eta_j}$ and let $x_j=\gamma_j(\ell_j)\in \overline \Sigma$ be the points where the geodesics intersect first time $\overline \Sigma$.
Then by Lemma \ref{measurement} (iv), we have 
        $$\textrm{grad}\;d^{+}_{x_j}(y)\bigg|_{y=\hat x_0 }=\frac{\dot{\gamma_j}(0)}{|\dot{\gamma_j}(0)|_{g}}=
        \frac{1}{|\eta_j|_{g}}\eta_j.
        $$
As vectors $\eta_j$, $j=1,2,\dots,n+1$ are linearly independent,  the point $\hat x_0 $  has a neighborhood $V\subset M$ so that the map

\beq\label{X coordinates}
\mathcal D:V\to \R^{n+1},\quad \mathcal D(y)=(d^{+}_{x_j}(y))_{j=1}^{n+1}
\eeq
defines regular coordinates on $M$ near $\hat x_0 $. 
By  (\ref{inverse exp}), $\mathcal D$ can be written in terms of the inverse functions of the exponential
functions and thus
the coordinates $(V,\mathcal D)$ are real-analytic coordinates of $M$. 
Let  $\Sigma_0=V\cap \Sigma$.

The above  construction of coordinates (\ref{X coordinates}) can be done
also on $\tilde M$. Thus we see that 
on $\tilde V=F(V)$ we have coordinates  $\tilde {\mathcal D}:\tilde V\to \R^{n+1},$ given by 
$\tilde {\mathcal D}(y)=(\tilde d^{+}_{\tilde x_j}(y))_{j=1}^{n+1}$, $\tilde x_j=F(x_j)$.
As $F$ is an isometry, we have ${\mathcal D}=\tilde {\mathcal D}\circ F$ on $V$.
Also, by (\ref{sigma isometry}), we have 
 ${\mathcal D}|_{\Sigma_0}=\tilde {\mathcal D}\circ  \Psi|_{\Sigma_0}$. These yield 
 \beq\label{F on Sigma}
 F|_{\Sigma_0}=\Psi|_{\Sigma_0}.
 \eeq
This in particular implies
 that $F|_{\Sigma_0}:\Sigma_0\to \tilde \Sigma_0=\Psi(\Sigma\cap V)$ is a 
diffeomorphism.

Recall that $F$  is an isometry and a real-analytic map. Thus we see that the unit
 normal vectors satisfy $\tilde \nu=F_*(\nu)$. As 
 $\tilde \nu=\Psi_*(\nu)$ by our assumptions,
   (\ref {F on Sigma}) yield
$F_*=\Psi_*$ on $T_{\hat x_0}M$.

 As $Z$ is a Killing field on $M$,
the field $\tilde Z_0:=F_*Z$ is a Killing field on $\tilde M$.
Our next aim is to show that $\tilde Z_0=\tilde Z$.
}



Now $\tilde Z_0 = F_* Z = \Psi_* Z = \tilde Z$ at $\hat x_0$.
As $F$ is an isometry, we have, see e.g. \cite[Prop. 3.59]{O},
$$
\tilde \nabla_{F_* X} \tilde Z_0 = F_* \nabla_X Z
$$
{\mltext for all vectors $X\in T_{\hat x_0}M$.} 
As
$F_* = \Psi_*$ at $T_{\hat x_0}M$, we have at $\Psi(\hat x_0)$,
$$
\tilde \nabla_{\Psi_* X} \tilde Z_0 = \tilde \nabla_{F_* X} \tilde Z_0 = F_* \nabla_X Z
= \Psi_* \nabla_X Z = \tilde \nabla_{\Psi_* X} \tilde Z.
$$
%
Hence $\tilde \nabla \tilde Z_0=\nabla \tilde Z$ at $\Psi(\hat x_0)$.
By  \cite{O}, see Lemma 9.27 and the text below it, the pair $(\tilde Z, A):=(\tilde Z,\nabla \tilde Z)$ of the Killing field and its covariant derivative satisfy a first order differential
equation over arbitrary smooth curve $\mu(s)$ on $\tilde M$ (on the original Riemannian versions of this
result,
see \cite{Konstant,Nomizu}),
\ba
& &\nabla_{\dot \mu(s)} \tilde Z( \mu(s))=-A( \mu(s)) \dot\mu(s),\\
& &\nabla_{\dot \mu(s)} A( \mu(s))= R(\tilde Z( \mu(s)), \dot\mu(s)),
\ea
where $R$ is the curvature operator of $\tilde M$. Thus, as 
$\tilde Z_0=\tilde Z$ and   $\nabla \tilde Z_0=\nabla \tilde Z$ at $\Psi(\hat x_0)$, we see that
 $\tilde Z_0=\tilde Z$ on the whole manifold $\tilde M$.

{\mltext Next we consider the scalar fields.}
 As $\phi$  is stationary with respect to $Z$
 and $Z$  is transversal to 
${\Sigma_0}$, we see that 
$\phi|_{\Sigma_0}$ determines $\phi$ in a neighborhood of ${\Sigma_0}$. 
Indeed, we see that $\phi(\exp_x(sZ))=\phi(x)=\tilde \phi(\exp_{\tilde x}(s\tilde Z))$
for $x\in {\Sigma_0}$, $\tilde x=F(x)$ and $s\in (-\epsilon(x),\epsilon(x))$, $\epsilon(x)>0$.  Thus see that $\tilde \phi\circ F=\phi$
in a neighborhood of ${\Sigma_0}$. As $F$  and the functions $\phi$ and $\tilde \phi$ are real-analytic, we have
 $\tilde \phi\circ F=\phi$ on the whole $M$. This proves the claim.}
\end{proof}



\medskip

\noindent{\bf Acknowledgements.} The authors express their gratitude to
the Mittag-Leffler Institute, where parts of this work have been done. The authors would like to thank Prof. Gunther Uhlmann for his generous support related to this work, and for suggesting the method used in the proof of Theorem 1. 

ML was partially supported by the Academy of Finland project 272312 and the Finnish Centre of Excellence in Inverse Problems Research 2012-2017.
YY was partially supported by the NSF grants DMS 1265958
and DMS 1025372. LO was partially supported by the EPSRC grant EP/L026473/1.

%


\begin{thebibliography}{aa}


\bibitem{And-D-H} L. Andersson,  M.  Dahl, and R. Howard, 
\emph{Boundary and lens rigidity of Lorentzian surfaces}, 
Trans. Amer. Math. Soc., {\bf 348} (1996), 2307--2329.


\bibitem{And} M. Anderson,  
\emph{On stationary vacuum solutions to the Einstein equations},
Annales Henri Poincare, {\bf 1} (2000), 977-994.


\bibitem{AKKLT} M. Anderson, A. Katsuda, Y. Kurylev, M. Lassas, and M. Taylor, \emph{Boundary regularity for the Ricci equation, Geometric Convergence, and Gelfand's Inverse Boundary Problem}, 
Invent. Math., {\bf 158} (2004), 261-321.


\bibitem{Beem} J. Beem, P. Ehrlich, and K. Easley, 
\emph{Global Lorentzian geometry}, 
Pure and Applied Mathematics, vol. 67,  Dekker, 1981. 


\bibitem{BeK} M. Belishev and Y. Kurylev, 
\emph{To the reconstruction of a Riemannian manifold via its spectral data (BC-method)},  
Comm. PDE,  {\bf 17}  (1992),  767--804.


\bibitem{BCG} G. Besson, G. Courtois, and S. Gallot, 
\emph{Entropies et rigidit\'{e}s des espaces localement sym\'{e}triques de courbure strictment n\'{e}gative},  
Geom. Funct. Anal., \textbf{5} (1995), 731--799.


\bibitem{Boyer} R. Boyer and R. Lindquist, 
\emph{Maximal Analytic Extension of the Kerr Metric}, 
J. Math. Phys., {\bf 8} (1967), 265-281.


\bibitem{BI} D. Burago and S. Ivanov, 
\emph{Boundary rigidity and filling volume minimality of metrics close to a flat one}, 
Ann. of Math., (2) \textbf{171} (2010), 1183--1211.


\bibitem{C1} C. Croke, 
\emph{Rigidity for surfaces of non-positive curvature}, 
Comment. Math. Helv., \textbf{65} (1990), 150--169.


\bibitem{CDS} C. Croke, N. Dairbekov, and V. Sharafutdinov, 
\emph{Local boundary rigidity of a compact Riemannian manifold with curvature bounded above}, 
Trans. Amer. Math. Soc., \textbf{352} (2000), no. 9, 3937--3956.


\bibitem{Chavel2006} I.~Chavel, 
\emph{Riemannian geometry: a modern introduction}, 
volume 98 of \emph{Cambridge Studies in Advanced Mathematics}, Cambridge University Press, Cambridge, second edition, 2006.


\bibitem{ChBook} Y. Choquet-Bruhat, 
\emph{General relativity and the Einstein equations},  
Oxford Univ. Press, 2009. 





\bibitem{DeTurck1981} D.~M. DeTurck and J.~L. Kazdan, 
\emph{Some regularity theorems in {R}iemannian geometry}, 
Ann. Sci. \'Ecole Norm. Sup. (4), \textbf{14(3)} 1981, 249--260.


\bibitem{Dz} V. Dzhunushaliev et al, 
\emph{Non-singular solutions to Einstein-Klein-Gordon equations with a phantom scalar field},
Journal of High Energy Physics {\bf 07} (2008) 094.



\bibitem{Eskin} G. Eskin,
\emph{Inverse hyperbolic problems and optical black holes}, 
Comm. Math. Phys., {\bf 297} (2010), 817--839. 



\bibitem{E2011} G. Eskin,
\emph{Artificial black holes}, 
Spectral theory and geometric analysis, 43-–53, Contemp. Math., {\bf 535}, Amer. Math. Soc., Providence, RI, 2011. 



\bibitem{FM} A. Fischer and J. Marsden, 
\emph{The Einstein evolution equations as a first-order quasi-linear symmetric hyperbolic system I.},
Comm. Math. Phys.,  \textbf{28}  (1972), 1--38.
%


\bibitem{Fridman} M.\ Fridman et al,
\emph{Demonstration of temporal cloaking},
Nature, {\bf 481} (2012), 62.


\bibitem{GLU}A. Greenleaf, M. Lassas, and G. Uhlmann,
\emph{On nonuniqueness for Calderon's inverse problem},  
Math. Res. Lett., {\bf 10} (2003),  685-693. 
%

\bibitem{GKLU1}
A. Greenleaf, Y. Kurylev, M. Lassas, and G. Uhlmann,
\emph{Full-wave invisibility of active devices at all frequencies},
Comm. Math. Phys., {\bf 275} (2007), 749-789.


\bibitem{GKLU2}
A. Greenleaf, Y. Kurylev, M. Lassas, and G. Uhlmann,
\emph{Invisibility and Inverse Problems}, 
Bull. Amer. Math. Soc., {\bf 46} (2009), 55-97.


\bibitem{GKLU3}
A. Greenleaf, Y. Kurylev, M. Lassas, and G. Uhlmann, 
\emph{Cloaking Devices, Electromagnetic Wormholes and Transformation Optics}, 
SIAM Review, {\bf 51} (2009), 3--33. 


\bibitem{Hartle} J. Hartle and S. Hawking,  
\emph{Solutions of the Einstein-Maxwell equations with many black holes}, Comm. Math. Phys., {\bf 26} (1972), 87--101.


\bibitem{HE} S. Hawking and G. Ellis,
\emph{The Large Scale Structure of Space-Time}, 
Cambridge Univ. press, 1973.


\bibitem{Helgason1962} S.~Helgason,
\emph{Differential geometry and symmetric spaces},
Pure and Applied Mathematics, Vol. XII. Academic Press, New York, 1962.


\bibitem {HKM}
T. Hughes, T. Kato, and J. Marsden,
\emph{Well-posed quasi-linear second-order hyperbolic systems with applications to nonlinear elastodynamics and general relativity},
Arch. Rational Mech. Anal. \textbf{63} (1976), 273--294.


\bibitem{Gr} M. Gromov, 
\emph{Filling Riemannian manifolds}, 
J. Diff. Geometry, \textbf{18} (1983), no. 1, 1--148.


\bibitem{Ionescu} A. Ionescu and S. Klainerman, 
\emph{On the local extension of Killing vector-fields in Ricci flat manifolds}, J. Amer. Math. Soc., {\bf 26} (2013), 563--593.


\bibitem{KaKu2} A. Katchalov and Y. Kurylev,
\emph{Multidimensional inverse problem with incomplete boundary spectral
data},
Comm. PDE, {\bf 23} (1998), 55--95.


\bibitem{KKL} 
A. Katchalov, Y. Kurylev, and  M. Lassas,  
\emph{Inverse boundary spectral problems}, 
Chapman-Hall/CRC, Boca Raton, FL, 2001. 


\bibitem{Konstant} B. Konstant, 
\emph{Holonomy and the Lie algebra of infinitesimal motions of a Riemannian manifold},
Trans. A.M.S., {\bf 80} (1955), 528-542.


\bibitem{KrKL} K. Krupchyk, Y. Kurylev, and M. Lassas, 
\emph{Inverse spectral problems on a closed manifold},
Journal de Mathematique Pures et Appliquees, {\bf 90} (2008), 42--59.


\bibitem{KLS-Einstein1} Y. Kurylev, M. Lassas, and G. Uhlmann, 
\emph{Inverse problems in spacetime I: Inverse problems for Einstein equations}, Preprint arXiv:1406.4776, 63 pp.


\bibitem{KLS-Einstein2} Y. Kurylev, M. Lassas, and G. Uhlmann, 
\emph{Inverse problems in spacetime II: Reconstruction of a Lorentzian manifold from light observation sets}, 
Preprint 	arXiv:1405.3386, 17 pp. 


\bibitem{LO} M.\ Lassas and L. Oksanen, 
\emph{Inverse problem for the Riemannian wave equation with Dirichlet data and Neumann data on disjoint sets}, 
Duke Math. J., {\bf 163} (2014), 1071-1103.


\bibitem{LTU} M.\ Lassas,  M.\ Taylor, and  G.\ Uhlmann,
\emph{The Dirichlet-to-Neumann map for complete Riemannian manifolds with
boundary}, 
Comm. Geom. Anal., {\bf 11} (2003), 207-222.


\bibitem{LSU} M. Lassas, V. Sharafutdinov, and G. Uhlmann, 
\emph{Semi-global boundary rigidity for Riemannian metrics}, 
Math. Ann., \textbf{325} (2003), 767--793.


\bibitem{LU} M.\ Lassas  and  G.\ Uhlmann,
\emph{Determining Riemannian manifold from boundary measurements},
Ann. Sci. \'Ecole Norm. Sup., {\bf 34} (2001), 771--787.


\bibitem{LeU} J. Lee and G. Uhlmann, 
\emph{Determining anisotropic real-analytic conductivities by boundary measurements},
Comm. Pure Appl. Math., {\bf 42} (1989), 1097--1112.


\bibitem{Le} U.\ Leonhardt,
\emph{Optical Conformal Mapping}, 
Science, {\bf 312} (2006), 1777-1780.


\bibitem{LeP} U.\ Leonhardt and T.\ Philbin, 
\emph{General relativity in electrical engineering}, 
New J.  Phys., {\bf 8} (2006), 247. 


\bibitem{McCall}  M. McCall et al, 
\emph{A spacetime cloak, or a history editor},
Journal of Optics {\bf 13} (2011), 024003. 


\bibitem{M} R. Michel, 
\emph{Sur la ridigit\'{e} impos\'{e}e par la longueur des g\'{e}od\'{e}siques}, Invent. Math., \textbf{65} (1981), 71--83.


\bibitem{Moffat} J. Moffat, 
\emph{Non-Singular Spherically Symmetric Solution in Einstein-Scalar-Tensor Gravity}, 
arXiv:gr-qc/0702070.


\bibitem{Muller} H. M\"uller zum Hagen, 
{\em On the analyticity of stationary vacuum solutions of Einstein's equation}, Proc. Cambridge Philos. Soc., {\bf 68} (1970), 199-201.


\bibitem{Morrey} C. Morrey, 
\emph{On the analyticity of the solutions of analytic non-linear elliptic systems of partial differential equations}, 
Am. J. Math., \textbf{80} (1958), 198-237.


\bibitem{O2} B. O'Neill, 
\emph{The geometry of Kerr black holes}, 
A K Peters, Ltd., 1995. 


\bibitem{O} B. O'Neill, 
\emph{Semi-Riemannian Geometry With Applications to Relativity}, 
Academic Press, New York, 1990.


\bibitem{PSS1} J.B.\ Pendry, D.\ Schurig, and D.R.\ Smith, 
\emph{Controlling electromagnetic fields},
Science,  {\bf 312} (2006), 1780-1782.


\bibitem{Nomizu} K. Nomizu,  
\emph{On local and global existence of Killing vector fields}, 
Ann. of Math., {\bf 72} (1960), 105--120.



\bibitem{PSU} G. Paternain, M. Salo and G. Uhlmann,  
\emph{Tensor tomography on surfaces}, 
Invent. Math. \textbf{193} (2013), no. 1, 229-247.



\bibitem{Ot} J. P. Otal, 
\emph{Sur les longuer des g\'{e}od\'{e}siques d'une m\'{e}trique a courbure n\'{e}gative dans le disque}, 
Comment. Math. Helv., \textbf{65} (1990), 334-347.


\bibitem{PU} L. Pestov and G. Uhlmann,
\emph{Two dimensional simple compact manifolds with boundary are boundary rigid}, Ann. of Math., \textbf{161(2)} (2005), 1089-1106.


\bibitem{SW} K. Sacks and H. Wu,  
\emph{General Relativity for Mathematicians}, 
Springer-Verlag, New York, 1977.


\bibitem{SU} P. Stefanov and G. Uhlmann, 
\emph{Lens rigidity with incomplete data for a class of non-simple Riemannian manifolds}, 
J. Diff. Geom., \textbf{82} (2009), 383-409.


\bibitem{SU1} P. Stefanov and G. Uhlmann, 
\emph{Rigidity for metrics with the same lengths of geodesics}, 
Math. Res. Lett., \textbf{5} (1998), 83-96.


\bibitem{SU3} P. Stefanov and G. Uhlmann, 
\emph{Boundary rigidity and stability for generic simple metrics}, 
J. Amer. Math. Soc., \textbf{18} (2005), 975--1003.


\bibitem{Tod} P. Tod,
\emph{Analyticity of strictly static and strictly stationary, inheriting and non-inheriting Einstein-Maxwell solutions}, 
Gen. Rel. Grav., {\bf 39} (2007), 1031--1042.


\bibitem{Visser} M. Visser, 
\emph{The Kerr spacetime - a brief introduction}, 
In: The Kerr spacetime (Ed.\ Wiltshire et al), Cambridge Univ. Press, 2009,  pp. 3-37.












\end{thebibliography}
\end{document}